\documentclass{amsart}
\usepackage{amsmath}
\usepackage{amssymb}
\usepackage[all]{xy}
\usepackage{comment}
\usepackage{color}

\theoremstyle{plain}
\newtheorem{thm}{Theorem}[section]
\newtheorem{thm*}{Theorem}[section]
\newtheorem{cor}[thm]{Corollary}
\newtheorem{prop}[thm]{Proposition}
\newtheorem{lemma}[thm]{Lemma}

\newtheorem{lemma*}{Lemma}

\newtheorem{construction}[thm]{Construction}

\theoremstyle{definition}
\newtheorem{defn}[thm]{Definition}
\newtheorem{remark}[thm]{Remark}

\newtheorem*{remark*}{Remark}

\newtheorem{ex}[thm]{Example}
\newtheorem{notation}[thm]{Notation}

\newtheorem{question*}{Question}

\numberwithin{equation}{thm}

\newcommand{\cN}{\mathcal N}
\newcommand{\cU}{\mathcal U}

\def\Spec{\operatorname{Spec}\nolimits}

\def\Coker{\operatorname{Coker}\nolimits}

\def\Proj{\operatorname{Proj}\nolimits}

\newcommand{\cKer}{\mathcal K\text{\it er}}
\newcommand{\cCoker}{\mathcal C\text{\it oker}}
\newcommand{\cIm}{\mathcal I\text{\it m}}

\newcommand{\bG}{\mathbb G}

\newcommand{\cO}{\mathcal O}
\newcommand{\cY}{\mathcal Y}
\newcommand{\bA}{\mathbb A}

\newcommand{\bP}{\mathbb P}
\newcommand{\bZ}{\mathbb Z}

\newcommand{\cC}{\mathcal C}

\newcommand{\cE}{\mathcal E}

\newcommand{\fp}{\mathfrak p}

\newcommand{\fg}{\mathfrak g}
\newcommand{\fh}{\mathfrak h}

\newcommand{\fu}{\mathfrak u}

\newcommand{\p}{\mathfrak p}
\newcommand{\q}{\mathfrak q}

\newcommand{\ol}{\overline}
\newcommand{\ul}{\underline}

\def\Spec{\operatorname{Spec}\nolimits}
\def\sl2{\operatorname{SL_{2(2)}}\nolimits}
\def\Ga2{\operatorname{\mathbb G_{\rm a(2)}}\nolimits}

\newcommand{\bH}{\mathbb H}
\newcommand{\bN}{\mathbb N}

\newcommand{\A}{\mathcal A}

\newcommand{\bu}{\bullet}

\date\today

\begin{document}

 \title[Invariants for $\bG_{(r)}$-modules]
 {Invariants for $\bG_{(r)}$-modules} 

%
 
 \author[ Eric M. Friedlander]
{Eric M. Friedlander$^{*}$} 

\address {Department of Mathematics, University of Southern California,
Los Angeles, CA 90089}
\email{ericmf@usc.edu}

\thanks{$^{*}$ partially supported by the Simons Foundation }



\begin{abstract}
We revisit the constructions given by J. Pevtsova and the author
of refined invariants for finite dimensional representations of infinitesimal group 
schemes $\bG_{(r)}$ over a field $k$ of characteristic $p>0$. 
Our focus is on the universal $p$-nilpotent operator seen as an element in the group 
algebra of the group scheme $\bG_{(r),X}$ over $X$, where $X$ is either the 
moduli space $V_r(\bG)$ of height $r$ $1$-parameter subgroups of $\bG$ or the
moduli space $\cC_r(\cN_p(\fg))$ of $r$-tuples of $p$-nilpotent, pair-wise commuting
elements of the Lie algebra of $\bG$.

We formalize Jordan type function using several variants of the continuous function
$JT_{\bG,r,M}(-): \bP V_r(\bG) \to \cY$ where $\cY$ is the poset of Young diagrams
with $p$-columns.  One of these variants is designed to be more conducive
to computation.  The vector bundle construction given by J. Pevtsova and
the author is extended to all finite dimensional $\bG_{(r)}$-modules, producing coherent 
sheaves on $X$ which are locally free on the strata of $X$ associated to $JT_{\bG,r,M}(-)$.

%
%
%
\end{abstract}

\maketitle


\section{Introduction}

In the present work, we utilize methods and results of support theory to investigate explicit 
invariants for specific finite dimensional modules  for certain classes of infinitesimal group schemes
over an arbitrary base field $k$ of characteristic $p$.   
This is in contrast to the role of support theory in establishing properties for various categories
of representations of  finite group schemes as in \cite{BIKP}.   Our techniques do not apply to finite groups,
for we utilize the support theory developed in collaboration with A. Suslin and C. Bendel
in \cite{SFB1}, \cite{SFB2} which replaces cohomological varieties by varieties of 1-parameter subgroups.

We further investigate the ``Jordan type function" introduced by J. Pevtsova 
and the author in \cite{FP2} and \cite{FP3}.
For a given affine group scheme $\bG$ of finite type over $k$ 
and a finite dimensional $\bG$-module $M$ for the $r$-th Frobenius kernel $\bG_{(r)}$ of $\bG$, 
we formulate the continuous function \ $JT_{\bG,r,M}(-): \bP V_r(\bG) \to \cY$  from the 
scheme-theoretic points of $\bP V_r(\bG)$ representing height $r$ 1-parameter subgroups 
to the partially ordered set $\cY$ of Young diagrams with $p$ columns.  By definition,
$JT_{\bG,r,M}(-) = JT_{\bG_{(r)},r,M}(-)$.  The continuity of 
$JT_{\bG,r,M}(-)$ implies that if a point $x \in \bP V_r(\bG)$ specializes
to $y  \in \bP V_r(\bG)$ then the Young diagram $JT_{\bG,r,M}(x)$ is greater or equal to
the Young diagram $JT_{\bG,r,M}(y)$.

This ``theory" $M \mapsto JT_{\bG,r,M}(-)$ constitutes a functor $JT_{\bG,r}: mod(\bG_{(r)}) \to 
Hom_{cont}(\bP V_r(\bG),\cY)$ from finite dimensional $\bG_{(r)}$-modules to continuous functions
$\bP V_r(\bG) \to \cY$ which captures much more information about a
$\bG_{(r)}$-module $M$ than does the cohomological support theory $M \mapsto \bP H^*(\bG_{(r)},k)_M$.
For example, if $p$ does not divide the dimension of $M$, then $\bP H^*(\bG_{(r)},k)_M$ equals 
$\bP H^*(\bG_{(r)},k)$ and therefore is independent of $M$, 
whereas $M \ \mapsto \ JT_{\bG,r,M}(-)$ distinguishes many such $\bG_{(r)}$-modules $M$ of a given 
dimension.  As might be expected, determination of $JT_{\bG,r,M}(-)$ is a challenge.  With this
in mind, we introduce a simplified Jordan type function $M \ \mapsto \ JT^{exp}_{\fg,r,M}(-)$
for affine group schemes $\bG$ which are of {\it exponential type of height $r$}.

After an elementary discussion of Jordan types, we provide a detailed formalism for the
construction of the ``universal $p$-nilpotent operator" $\Theta_{\bG,r}  \in 
k[V_r(\bG)]\otimes k\bG_{(r)} \equiv k\bG_{(r),V_r(\bG)}$ of height $r$
constructed by J. Pevtsova and the author in \cite{FP3}.   In doing so, we expand upon the 
connection between 1-parameter subgroups for $\bG$ and operators on $\bG_{(r)}$-modules.
The operator $\Theta_{\bG,r}$ acting on $k[V_r(\bG)]\otimes M$ determines the Jordan 
type function $JT^{exp}_{\fg,r,M}(-)$.
The categorically inclined reader might want to categorify the functor $JT_{\bG,r}$
by replacing $Hom_{cont}(\bP V_r(\bG),\cY)$ by the category of graded $k[V_r(\bG)]$-modules
equipped with a $p$-nilpotent homogeneous operator.

In order to make the Jordan type function more explicit and computationally accessible,
we restrict our attention to affine group schemes $\bG$ which are of
{\it exponential type of height $r$}.  By definition, such groups $\bG$ are equipped with an 
exponential map with the property that   height $r$ 1-parameter subgroups of 
$\bG$ are given by applying the exponential map to $r$-tuples of $p$-nilpotent, pair-wise 
commuting elements of the Lie algebra $\fg$ of $\bG$.   Thus, the ``moduli space"
for height $r$ 1-parameter subgroups of $\bG$ (of exponential type $r$) is the variety $\cC_r(\cN_p(\fg))$
of such $r$-tuples of elements in $\fg$.  For $\bG = GL_N$, the associated exponential
map is the usual (truncated) exponential sending a $p$-nilpotent element of $A\otimes \fg l_N$ to
its exponential in $GL_N(A)$ for any commutative $k$-algebra $A$.

For a given finite dimensional $\bG_{(r)}$-module $M$ with $\bG$ of exponential type
of height $r$, the Jordan type function becomes a continuous
function $JT_{\fg,r,M}(-): \bP \cC_r(\cN_p(\fg)) \to \cY$ associated to the ``projectivization" 
of the universal $p$-nilpotent operator $\Theta_{\fg,r} \in k[\cC_r(\cN_p\fg))]\otimes k\bG_{(r)}$
corresponding to $\Theta_{\bG,r}$.  Our exponential Jordan type function 
$JT_{\fg,r,M}^{exp}(-):  \bP \cC_r(\cN_p(\fg)) \to \cY$ is defined by  replacing $\Theta_{\fg,r}$
by its linear part (denoted $\Theta_{\fg,r}^{exp}$) with respect to a natural grading on 
the group algebra $k\bG_{a(r)}$ of $\bG_{a(r)}$.  In the very special case that $\bG$ equals $\bG_{a(r)}$, this 
simplification is related to J. Carlson's use of cyclic shifted subgroups to identify support
varieties for $(\bZ/p)^{\times r}$-modules (see \cite{C}).   

The suitability of $\Theta_{\fg,r}^{exp}$ follows from work of Pevtsova, Suslin,
and the author \cite{FPS} and subsequent work of P. Sobaje \cite{Sob}.
Namely,  for a $\bG_{(r)}$-module $M$ of dimension $m$ and a point of $x \in \bP \cC_r(\cN_p(\fg)$),
$JT_{\fg,r,M}(x) \ = \ m/p\cdot [p]$ if and only if $JT^{exp}_{\fg,r,M}(x) \ = \ m/p\cdot [p]$.
Thus, the replacement of $JT_{\bG,r,M}(-)$ by $JT_{\fg,r,M}^{exp}(-): \bP \cC_r(\cN_p(\fg)) \to \cY$ 
does not affect (refined) support varieties considered in previous work of Pevtsova 
and the author in \cite{FP2} and \cite{FP3}.  
We provide sample computations which demonstrate
the accessibility of $JT_{\fg,r,M}^{exp}(-)$ for certain $\bG_{(r)}$-modules $M$.

For $\bG$ of exponential type of height $r$, we verify that vector bundles on $\bP \cC_r(\cN_p(\fg))$
associated to $\bG_{(r)}$-modules $M$ of constant Jordan type using $\Theta_{\fg,r}$
(naturally corresponding to vector bundles previously constructed in \cite{FP3})
 are $\bA^1$-homotopy equivalent to vector bundles using $\Theta_{\fg,r}^{exp}$.
We extend these constructions of vector bundles to apply to arbitrary finite dimensional 
$\bG_{(r)}$-modules $M$, yielding coherent sheaves on $\bP \cC_r(\cN_p(\fg))$
whose restrictions on strata associated to the Jordan type function for $M$ are locally free.
Unlike the situation for modules of 
constant Jordan type, these stratified coherent sheaves associated to a given finite
dimensional $\bG_{(r)}$-module will typically depend upon  whether we use the Jordan
type function associated to $ \Theta_{\bG,r}$ or associated to $\Theta_{\fg,r}^{exp}$.

As we observe in the last section, the issue  of ``stabilization with respect to $r$" of Jordan type 
functions of height $\leq r$ is closely related to the Jordan type function for finite 
dimensional rational $\bG$-modules introduced in \cite{F15} if 
$\bG$ is of exponential type of infinite height.  This corrects a minor error of \cite[Prop 4.8]{F15}
which incorrectly formulated such a stabilization result.

We thank Julia Pevtsova and Paul Sobaje for their insights into the mathematics we present.
\vskip .2in


\section{Jordan types and specialization}
\label{sec:JT}

We begin by recalling some basic properties of Jordan forms of $p$-nilpotent endomorphisms.

\begin{defn}
\label{defn:Jordan-form}
Let $M$ be a $k[t]/t^p$-module of dimension $m$ given by the map $\rho_M: k[t]/t^p \to M_{m,m}(k)$
of $k$-algebras.  Let $\ul a = (a_1,\ldots,a_p) \in \bN^{\times p}$ (i.e., a $p$-tuple of non-negative
integers) be defined by setting $a_i$ to be the
number of blocks of size $i$ of the Jordan canonical form of the matrix $\rho_M(t)$.  We say that
the Jordan type of $M$ is $\ul a$, expressing this as $JT(M) = \ul a$ or as 
$JT(M) = a_p[p] + a_{p-1}[p-1] + \cdots + a_1[1]$.
\end{defn}

\vskip .1in

The following proposition recalls familiar interpretations of Jordan types.

\begin{prop}
\label{prop:Young}
Let $M$ be a $k[t]/t^p$-module of dimension $m$ such that $JT(M) = \ul a $.
\begin{enumerate}
\item
As a $k[t]/t^p$-module, $M$ is isomorphic to \ $\bigoplus_{i= 1}^p (k[t]/t^i)^{\oplus a_i}$. 
\item
The partition of $m = \sum_{i=1}^p a_i\cdot i$ with $a_i$ subsets of size $i$ 
is naturally associated to $\ul a =  (a_1,\ldots,a_p)$.
\item
The usual partial ordering of partitions determines the
partial ordering on $\bN^{\times p}$ given by
$$\ul a \ \leq \ \ul b \quad \iff \quad \forall s, 1 \leq s \leq p, \quad
(\sum_{i=s}^p a_i\cdot i) \leq (\sum_{i=s}^p b_i\cdot i).$$
\item
Uniquely associated to $\ul a$ is a Young diagram with $p$-columns such that
the diagram has $a_i$ rows with $i$ boxes and the rows of the diagram are
 left-adjusted with descending length.  Thus, the partial ordering 
$\ul a \ \leq \ \ul b$ requires for each $\ell$ that the number of boxes in the top $\ell$ 
rows of $\ul a$ is less or equal to the number of boxes in the top $\ell$ rows of $\ul b$.
\item
The Jordan type of $t^j$ acting on $M \simeq  \oplus_{i= 1}^p (k[t]/t^i)^{\oplus a_i}$
(denoted by $\ul a^j \in \bN^{\times p}$)
has associated Young diagram obtained from the Young diagram for $\ul a$
by removing the first (i.e., left-most) $j-1$ columns of $\ul a$ and adding rows with 
each having only 1 box so that the number of boxes of $\ul a^j$ equals $m$.
 \end{enumerate}
\end{prop}

\vskip .1in

The lemma below recalls some properties of the partial ordering on Jordan types.
We denote by $rk(t^j:M \to M)$ the dimension of the subspace
$t^j(M) \subset M$, the dimension of the image of multiplication by $t^j$ on the vector space $M$.

\begin{prop}
\label{prop:rank}
Consider $m$-dimensional $k[t]/t^p$-modules $M, \ N$  isomorphic to \ 
$\bigoplus_{i= 1}^p (k[t]/t^i)^{\oplus a_i}, \ \bigoplus_{i= 1}^p (k[t]/t^i)^{\oplus b_i}$. 
As above, let $\ul a^s$ (respectively, $\ul b^s$) denotes the Jordan type of 
$t^s$ acting on $M$ (resp, $N$) for any $s, \ 1 \leq s < p$.
\begin{enumerate}
\item
 $rk(t^s:M \to M) \ = \ \sum_{i=s+1}^p a_i(i-1-s)$.
 \item
 For all $s, \ 1 \leq s < p$, \ $\ul a^s \ = \ \sum_{i=s+2}^p a_i[i-s] + (m - rk(t^s:M \to M))[1]$.
 \item
 $\ul a \leq \ul b \ \iff \ul a^s \leq \ul b^s, \ \forall s, 1 \leq s < p$.
 \item 
 For all $s, \ 1 \leq s < p$, \ $\ul a^s \ \leq \ul b^s \ \iff rk(t^s:M \to M) \leq rk(t^s:N \to N).$
 \end{enumerate}
\end{prop}

\begin{proof}
Assertions (1) and (2) are verified by an easy computation.  Assertion (3) follows from assertion (2)
and Proposition \ref{prop:Young}(3).  Assertion (4) follows easily from assertion (2).
\end{proof}

\vskip .1in

Since a surjective map $f: M \to N$  of  $k[t]/t^p$-modules
induces a surjective map $f: t^j(M) \to t^j(N)$ 
for any $j$ with $ 1 \leq j < p$,  Proposition \ref{prop:rank}(4)
 implies the following additional corollary.

\begin{cor}
\label{cor:surjective}
Let $f: M \to N$ be a surjective map of finite dimensional $k[t]/t^p$-modules.
Then \ $JT(N) \leq JT(M)$, with equality if and only if $f$ is an isomorphism.
\end{cor}

\vskip .1in

We formulate the partially ordered abelian monoid  $\cY_\leq$ of ``Young diagrams"
equipped with a (tensor) product.

\begin{defn}
\label{defn:poset}
Denote by $\cY_\leq$ the monoid  $\bN^{\times p}$ provided with the 
 partial ordering of Proposition \ref{prop:Young}(3), and 
denote by $\cY_{\leq, n}$ the partially ordered subset of $\cY_\leq$
whose objects are $p$-tuples $\ul a$ with $\sum_{i=1}^p a_i\cdot \ \leq \ n$.

We equip $\cY_\leq$ with the bi-additive pairing
$$(-)\otimes (-): \ \cY_\leq \times \cY_\leq \quad \to \quad \cY_\leq \ ,$$
where $\ul a \otimes \ul b$ is the $p$-tuple $\ul c$ defined by the isomorphism
$$\bigoplus_{i= 1}^p (k[t]/t^i)^{\oplus c_i})  \quad  \simeq \quad
(\bigoplus_{i= 1}^p (k[t]/t^i)^{\oplus a_i}) \ \bigotimes \ (\bigoplus_{i= 1}^p (k[t]/t^i)^{\oplus b_i})$$
of $k[t]t^p$-modules (see \cite[Cor10.3]{CFP}).
\end{defn}

\vskip .1in

We recall the lower semi-continuity of Jordan type, an easy consequence of 
Nakayama's Lemma (see \cite[Prop 3.2]{FPS}). 

\begin{prop} 
\label{prop:specializeM}
Consider a 
noetherian, local domain $R$ with field of fractions $L$ and residue field $K$.
Let $M_R$ be a free $R$-module of finite rank equipped with a $p$-nilpotent 
endomorphism $\theta_R \in End_R(M_R,M_R)$.  Denote $K\otimes_R M$ by $M_K$
and $L \otimes_R M$ by $M_L$, and set 
$$\theta_K = K \otimes_R \theta_R \in End_K(M_K,M_K), \quad  
\theta_L = L \otimes _R \theta_R \in End_L(M_L ,M_L).$$
 So defined, $M_K$ is a $K[t]/t^p$-module with the action of $t$
given by $\theta_K$ and $M_L$  is an $L[t]/t^p$-module with the action of $t$
given by $\theta_L$.
We say that the $L[t]/t^p$-module $M_L$ specializes along $R$ to the $K[t]/t^p$-module 
$M_K$.  

For such a specialization, \ $JT(M_K) \ \leq \ JT(M_L)$.
\end{prop}

\begin{proof}
By applying Nakayma's Lemma \cite[Lem 4.3]{Lang} to the cokernels of the 
maps $(\theta_R)^i: M_R \to M_R$,
we conclude that $dim( coker(\theta_L^i)) \leq dim( coker(\theta_K^i))$
so that $rk(\theta_K^i) \leq rk(\theta_L^i)$ for each $i,\ 1 \leq i < p$.  Thus, the proposition
follows from Proposition \ref{prop:rank}.  
\end{proof}

\vskip .1in

\begin{defn}
\label{defn:specializeX}
Let $X$ be a scheme of finite type over $k$.  For two scheme-theoretic points $x, y$
of $X$, we write $x \leq y$ if $x$ is in the (Zariski) closure of $y$.   This endows 
the set of scheme-theoretic points of $X$ with a natural partial ordering.

If $\Spec R \subset X$
is an affine open subset of $X$ containing $x$, then the scheme-theoretic point $x \in \Spec R$
corresponds to the prime ideal $\p_x \subset R$.   The structure sheaf $\cO_X$ of $X$ has stalk
the local ring $R_{(\p_x)}$ with residue field $k(x) = R_{(\p_x)}/\p_x R_{(\p_x)}$.  We identify the 
scheme-theoretic point $x \in X$ with the associated map of schemes $\chi_x: \Spec k(x) \to X$.

\end{defn}
\vskip .1in

\begin{lemma}
\label{lemma:special}
Let $x, y$ be two scheme-theoretic points of $X$. Then $x \leq y$ if and only if
there exists a local domain $R$ with field of fractions
$k(y)$ and residue field $k(x)$ and a morphism $\Spec R \to X$ whose restriction to 
$\Spec k(y)$ is $\chi_y: \Spec k(y) \to X$ and whose restriction to $\Spec k(x)$ is 
$\chi_x: \Spec k(x) \to X$.
\end{lemma}

\begin{proof}
We may assume $x \not= y$.  Let $\Spec A$ be a Zariski open subset of $X$ containing $x$
with corresponding prime $\p_x \subset A$.  Since $x$ is in the closure of $y$, the prime ideal
$\p_y$ is contained in $\p_x$.  Set $S$ equal to the quotient $A/\p_y$ and let
 $R$ be the localization of $S$ at (the image of ) $\p_x$.
So defined, $R$ is a local domain with residue field $k(x)$
and field of fractions $k(y)$.  Moreover, $\chi_x: \Spec k(x) \to X$ and $\chi_y: \Spec k(y) \to X$
factor through $\Spec R \subset \Spec S \subset \Spec A \subset X$.

Conversely, assume given a local domain $R$ with field of fractions
$k(y)$ and residue field $k(x)$ and a morphism $\Spec R \to X$ whose restriction to 
$\Spec k(x)$ is $\chi_x: \Spec k(x) \to X$ and whose restriction to $\Spec k(y)$ is 
$\chi_y: \Spec k(y) \to X$.  Then the kernel of 
$R \to k(x)$  equals the prime ideal $\p_x$ and contains $\p_y$,  the kernel of  $R \to k(y)$.
Thus, $x \leq y$.
\end{proof}

\vskip .1in

The following proposition will be employed 
with $A = k[V_r(\bG)]$ as in Definition \ref{defn:VrG} and with the action of $t$
given by various powers of the universal  $p$-nilpotent operator of 
Theorem \ref{thm:theta}.   We point out that in this proposition we consider the
tensor product over the commutative algebra $A$ of $A$-modules; in discussion
below, we shall frequently consider tensor product over $k$ of $\bG$-modules
using the Hopf algebra structure of $k[\bG]$.

\begin{prop}
\label{prop:JM(-)}
Let $A$ be a Noetherian commutative $k$-algebra 
and consider a finitely generated, projective $A$-module $M_A$
equipped with a $p$-nilpotent 
endomorphism $\theta_A \in Hom_A(M_A,M_A)$.  Then sending a scheme-theoretic
point $x \in \Spec A$ (i.e., a prime ideal $\p_x \subset A$) to the 
Jordan type of $k(x) \otimes_A M_A$ 
as a $k(x)[t]/t^p$-module with $t$ acting as $k(x) \otimes \Theta_A \in Hom_{k(x)}(M\otimes k(x),M\otimes k(x))$
determines a map of partially ordered sets 
\begin{equation}
\label{eqn:JTM}
JT_{A,M_A}(-): (\Spec A) \quad \to \quad \cY_\leq \ , \quad x \mapsto JT(k(x) \otimes_A M_A).
\end{equation}

If $M_A, N_A$ are finitely generated $A[t]/t^p$-modules, then for all
$x \in \Spec A$
$$JT_{A,M_A}(x)\oplus JT_{A,N_A}(x) \quad = \quad JT_{A,M_A \oplus N_A}(x), $$
$$JT_{A,M_A}(x)\otimes JT_{A,N_A}(x) \quad = \quad JT_{A,M_A\otimes_A N_S}(x).$$
\end{prop}

\begin{proof}
The assertion that $JT_{A,M_A}(-)$ is a map of posets is the assertion
that $x \leq y \in \Spec A$ implies that $JT_{A,M_A}(x) \leq JT_{N,N_A}(y)$.
As in Lemma \ref{lemma:special}, we may restrict $M_A$ along a map from $A$
to  a local domain $R$ with field of fractions $k(y)$ and residue field $k(x)$.  
Thus, we reduce to the case that $A = R$ is a local domain with field of fractions 
$k(y)$ and residue field $k(x)$; in this case, $M_A$ is a free $A$-module of
finite rank.  Thus, Proposition 
\ref{prop:specializeM} applies to prove $JT_{A,M_A}(x) \leq JT_{A,M_A}(y)$.

Observe that specialization along a scheme-theoretic point of $\Spec A$ commutes 
with direct sums and tensor products.  Thus, the
 fact that the function $JT_{A,-}(x)$ (sending a finitely generated $A[t]/t^p$-module $M_A$ to 
$JT_{A,M_A}(x) \in \cY_\leq$) respects both addition and tensor product is 
immediate from the definitions of addition and tensor product for $\cY_\leq$.
\end{proof}

\vskip .1in

Proposition \ref{prop:JM(-)} does not immediately imply that 
$JT_{A,M_A}(-): \Spec A \ \to  \ \cY_{\leq}$ is continuous when $\Spec A$
is given the Zariski topology.   This is because arbitrary unions of closed subsets
(i.e., subsets closed under specialization)
are often not closed in the Zariski topology.  We supplement Proposition \ref{prop:JM(-)}
with another consequence of Nakayama's Lemma in commutative algebra.
(It seems more convenient for this lemma to refer to the scheme-theoretic points of
$\Spec A$ as prime ideals $\p \subset A$.)

\begin{lemma}
\label{lem:ranks}
Let $A$ be a Noetherian commutative $k$-algebra 
and consider a projective $A$-module $M_A$ of rank $m$
equipped with a $p$-nilpotent endomorphism $\theta_A:  M_A \to M_A$.
Then for any $d > 0$, 
\begin{equation}
\label{eqn:subrank}
C_d \ \equiv \ \{ \p \in \Spec A : \ rk(\theta_\fp: k(\fp) \otimes_A M_A \to k(\fp) \otimes_A M_A) \leq d \}
\end{equation}
is a closed subset of $\Spec R$.
\end{lemma}

\begin{proof}
By considering the intersection of $C_d$ with each member of an affine open covering of $\Spec A$
restricted to which $M_A$ is free, it suffices to prove the lemma assuming $M_A$ is free.  Moreover, 
if $\q_1, \ldots, \q_r$ are the minimal primes of $\Spec A$, then the closures of $\{ \q_i \}$ are the
(closed) irreducible components of $\Spec A$.  Thus, by considering the intersection of $C_d$ with
each of these closures, it suffices to assume that $A$ is a domain.

We view $\theta_A$ as an element of $M_{m,m}(A)$.  Then $C_d$ is the closed subset defined
as the zero locus of the determinants of all $(d+1\times d+1)$-submatrices of $\theta_A$.
\end{proof}

We remind the reader that an Alexandrov discrete topology on a $T_0$ space
is a topology in which the arbitrary union of closed subsets is closed.
(See \cite{Alex} and \cite{John}.)   The topology on $\cY_\leq$ which we utilize is such a topology.
\vskip .1in

\begin{defn}
\label{defn:top}
We give $\cY_\leq$ the topology whose closed subsets are those
subsets $S \subset \cY_\leq$ satisfying the condition that if $s \in S$ and $s^\prime \leq s$
then $s^\prime \in S$.  

In particular, $S \subset \cY_\leq$ is closed if and only if
$S \cap \cY_{\leq,m}$ is closed in $\cY_{\leq,m}$ for all $m > 0$.
\end{defn}

\vskip .1in

\begin{thm}
\label{thm:cont}
Let $A$ be a Noetherian commutative $k$-algebra 
and consider a projective $A$-module $M_A$ of rank $m$
equipped with a $p$-nilpotent endomorphism $\theta_A$.  As usual, we give $\Spec A$ 
the Zariski topology.  Then 
$$JT_{A,M_A}(-): \Spec A \quad \to  \quad \cY_{\leq, m} \quad \hookrightarrow \quad \cY_\leq$$
is continuous.
\end{thm}

\begin{proof}
A subset $C \subset \cY_{\leq, m} $ is closed if and only 
$$C  \quad = \quad \bigcup_{\ul a \in C}\ol{\{ \ul a \}},$$
where $\ol{\{ \ul a \}}$ is the closure of the singleton subset $\{ \ul a \}$.
Thus, it suffices to prove that $JT_{A,M_A}^{-1}(\ol{\{ \ul a \}})$ is closed
for all $\ul a \in \cY_{\leq, m} $ (since $\cY_{\leq, m}$ is finite).  By definition, $JT_{A,M_A}^{-1}(\ol{\{ \ul a \}})$
consists of those $x \in \Spec A$ such that $JT_{A,M_A}(x) \ \leq \ \ul a$ in  
$\cY_{\leq, m}.$  
By Proposition \ref{prop:rank} (4), $JT_{A,M_A}(x) \ \leq \ \ul a$ if and only if the 
rank of $(\theta_x)^s: k(x) \otimes_A M_A \to k((x) \otimes_A M_A$ is less
than or equal to the rank of $t^s$ on $\oplus_{i=1}^p (k[t]/t^i)^{\oplus a_i}$ 
(which equals $\sum_{i=s+1}^p a_i\cdot (i-s)$) for all $s$, $1 \leq s < p$.

Applying Lemma \ref{lem:ranks} to each of $\theta^s, 1 \leq s < p$, we conclude
the set of those $x \in \Spec A$ such that $JT_{A,M_A}(x) \ \leq \ \ul a$ is 
closed in $\Spec A $.
\end{proof}

\vskip .1in

\begin{remark}
Considering stable Jordan types (i.e., disregarding blocks of size $p$) can be 
useful, but has the disadvantage that the stable version of $JT_{A,M_A}(-)$ 
 is not continuous.
\end{remark}

\vskip .1in

\begin{notation}
In the remainder of this paper, we shall leave implict the subscript $\leq$ on $\cY_{\leq,m}$
and $\cY_\leq$.
\end{notation}

\vskip .2in


\section{Jordan type functions for $\bG$-modules}
\label{sec:functors}

Consider an affine group scheme $\bG$ of finite type over $k$. We are interested in
formulating invariants for finite dimensional $\bG_{(r)}$-modules.
Here, $\bG_{(r)}$ is the kernel of the $r$-th iterate of the Frobenius map 
$F: \bG \to \bG^{(1)}$, where $k[\bG^{(1)}]$ is the base change of $k[\bG]$ along the $p$-th power
map $(-)^p: k \to k$.  (See \cite[\S 1]{FS}.)   Thus, 
$$\bG_{(r)} \ \equiv \ ker\{ F^r: \bG \to \bG^{(r)} \}$$
is an infinitesimal group scheme of height $ \leq r$ with coordinate algebra $k[\bG_{(r)}]$.  
We denote by $k\bG_{(r)}$ the $k$-linear dual of the finite dimensional Hopf algebra
$k[\bG_{(r)}]$.

For any commutative $k$-algebra $R$, we shall employ various notations (depending upon
context) for 
$$\bG_R \quad \equiv \quad \bG \times \Spec R \quad \equiv \quad R \otimes \bG$$
for  the group scheme defined as the base change along $k \to R$ 
of the affine group scheme $\bG$ over $k$.  Thus, the coordinate algebra of $\bG_R$ equals
$R\otimes k[\bG]$.   If $\bG$ is a finite group scheme over $k$, then the ``group algebra" of $\bG_R$ 
(namely, $Hom_R(R\otimes k[\bG],R))$ is identified with $R \otimes k\bG $ 
by sending $1\otimes \phi \in R \otimes k\bG$ to $id_R \cdot \phi(-):R\otimes k[\bG] \to R$.

\vskip .1in

\begin{defn}
\label{defn:VrG}
Let $\bG$ be an affine group scheme of finite type
over $k$ and let $r$ be a positive integer.  As proved in \cite[Thm1.5]{SFB1}, the functor of 
commutative $k$-algebras sending an algebra $A$ to the set of
maps of group schemes $\mu: \bG_{a(r),A} \to \bG_A$ over $A$ (i.e., the set of height $r$
1-parameter subgroups of $\bG_A$) is 
representable.  The commutative $k$-algebra representing this functor is denoted $k[V_r(\bG)]$, its
spectrum by  $V_r(\bG)$.

Given a $\bG_{(r)}$-module $M$ and a scheme theoretic point $x \in V_r(\bG)$, base change
provides a $\bG_{(r),k(x)}$-module structure  on $k(x) \otimes M$.  Namely,
if $\mu_x: \bG_{a(r),k(x)} \to \bG_{(r),k(x)}$ is the 1-parameter subgroup over $k(x)$ parametrized
by $\chi_x: \Spec k(x) \to V_r(G)$, then restricting along $\mu_x$ determines
 a $\bG_{a(r),k(x)}$-module structure on $k(x) \otimes M$.

In other words, restriction along the map of  ``group algebras" $\mu_{x,*}: k(x) \otimes k\bG_{a(r)} \to 
k(x) \otimes k\bG_{(r)}$ determines the $k(x) \otimes k\bG_{a(r)}$-module structure on $k(x) \otimes M$.
\end{defn}

\vskip .1in

The following definition makes explicit how the map of group algebras 
$\mu_{x,*}: k(x) \otimes k\bG_{a(r)} \to k(x) \otimes k\bG_{(r)}$ 
naturally determines by ``restricting along $\epsilon_r: k[u]/u^p \to k\bG_{a(r)}$"
a map of $k(x)$-algebras 
$(\mu_x)_* \circ \epsilon_r:  k(x) \otimes k[u]/u^p  \to \   k(x) \otimes k\bG_{(r)}$.
We remind the reader that the group algebra $k\bG_{a(r)}$ (i.e., the dual of the 
coordinate algebra $k[\bG_{a(r)}] = k[t]/t^{p^r}$) can be described as the commutative 
$k$-algebra $k[u_0,\ldots,u_{r-1}]/(u_i^p)$ where $u_i$ is the $k$-linear dual 
to $t^{p^i}$.

\vskip .1in

\begin{defn}
\label{defn:epsilon-r}
Let $\bG$ be an affine group scheme of finite type over $k$.
For $r \geq 1$,   we denote by $\epsilon_r: k[u]/u^p \to k\bG_{a(r)}$  the map of $k$-algebras
(but not of Hopf algebras if $r > 1$) which sends $u$ to the element  $u_{r-1} \in k\bG_{a(r)}$. 
A scheme-theoretic point $x \in  V_r(\bG)$ corresponding to the 1-parameter subgroup
$\mu_x: \bG_{a(r),k(x)} \to \bG_{(r),k(x)}$  determines the map of $k(x)$-algebras 
\begin{equation}
\label{eqn:lower*}
(\mu_x)_* \circ \epsilon_r:  k(x)\otimes k[u]/u^p \to k(x) \otimes k\bG_{a(r)} \to 
k(x) \otimes k\bG_{(r)},  \ u \mapsto (\mu_x)_*(1\otimes u_{r-1}).
\end{equation}

Thus, if $M$ is a $\bG$-module, then $(\mu_x)_* \circ \epsilon_r)^*(k(x) \otimes M)$ is equipped with a 
$k(x)\otimes k[u]/u^p$-module structure such that $u$ acts as $(\mu_x)_*(1\otimes u_{r-1})$ on $k(x) \otimes M$.
\end{defn}

\vskip .1in

To unpack $(\mu_x)_*(u_{r-1}) \in k(x) \otimes k\bG_{(r)}$, we observe that this is the
$k(x)$-linear function sending 
$f \in k[\bG]$ to the coefficient of $t^{p^{r-1}}$ in the expansion 
of $\mu_x^*(f) \in k(x) \otimes k[\bG_{a(r)}] = k(x)[t]/t^{p^r}$.

\vskip .1in

We recall the local Jordan type of a finite dimensional $\bG_{(r)}$-module as first formulated
in \cite{FP2} and \cite{FP3} by Pevtsova and the author.

\begin{defn} (See \cite[Cor 3.8]{FP3}.)
\label{defn:local-Jordan}
Let $\bG$ be an affine group scheme of finite type over $k$, $r \geq 1$, and $M$ a finite dimensional
$\bG_{(r)}$-module.  Then the local Jordan type of $M$ at a scheme-theoretic point $x  \in V_r(\bG)$,
\ $JT_{\bG,r,M}(x)$, \
is defined to be the Jordan type of the $k(x)[u]/u^p$-module $((\mu_x)_* \circ \epsilon_r)^*(k(x) \otimes M)$.
\end{defn}

\vskip .1in

In the following definition, we recall 
the natural grading on $V_r(\bG)$ as discussed  in \cite[\S 1]{SFB1}.  This is particularly
relevant because there is a natural morphism $\Proj k[V_r(\bG)] \to \bP H^*(\bG_{(r)},k) \simeq \Pi(\bG_{(r)})$
which is a homeomorphism of Zariski spaces
to the $\pi$-point scheme  $\Pi(\bG)$ (see \cite[Prop 3.1]{FPS}).  

\begin{defn}
\label{defn:grading}
The natural right action $V_r(\bG_a) \times \bA^1 \to V_r(\bG_a)$ sends \\
$(\nu: \bG_{a(r),A} \to \bG_{a(r),A}, \alpha \in A)$ to $\alpha\cdot \nu: \bG_{a(r),A} \to \bG_{a(r),A}$ defined by $(\alpha\cdot \nu)^*(t) = \nu^*(\alpha\cdot t)$.
Composing this  with 1-parameter subgroups $\mu: \bG_{a(r),A} \to \bG_{(r),A}$ defines
a right action $V_r(\bG) \times \bA^1 \to V_r(\bG)$.  This action corresponds to a grading 
on  the commutative algebra $k[V_r(\bG)]$.  

We set 
$$\bP V_r(\bG) \quad \equiv \quad \Proj k[V_r(\bG)].$$
If $\mu: \bG_{(r),R} \to \bG_R$ is an $R$-point of $V_r(\bG)$ for some commutative $k$-algebra $R$
 and $s$ an $R$-point of $A^1$ (i.e., an element $s \in R$),
then  $(s \circ \mu)_*(u_{j}) = s^{p^i}\mu_*(u_{j})$ for $u_j\in k\bG_{a(r)} = k[u_0,\cdots,u_{r-1}]/(u_i^p).$
(See the proof of \cite[Prop 6.5]{SFB2}).)
\end{defn}

\vskip .1in

We rephrase Lemma 1.12 of \cite{SFB1} in order to elaborate upon this grading of $V_r(\bG)$.

\begin{prop}
\label{prop:grading}
The commutative algebra $k[V_r(GL_N)]$ is a graded quotient of the symmetric algebra over $k$
generated with generators 
$\{ X_{i,j}^\ell; 1 \leq i,j, \leq N, 0 \leq \ell < r \}$ provided that $X_{i,j}^\ell$ is given degree $p^\ell$.
Moreover, if $\bG \subset GL_N$ is a closed immersion, then $k[V_r(\bG)]$ is a graded quotient 
of $k[V_r(GL_N)]$.

Consequently, for any affine group scheme $\bG$ of finite type over $k$, $k[V_r(\bG)]$ is a 
graded commutative algebra generated by homogeneous elements of degrees $p^i, \ 0 \leq i < r$.
\end{prop}

The following theorem summarizes fundamental results of Suslin, Bendel, and the author 
which relate 1-parameter subgroups of $\bG$ of height $r$ to cohomological support varieties of $\bG_{(r)}$.

\begin{thm}
\label{thm:fundamental} \cite[Thm 5.2]{SFB1}, \cite[Thm 6.7]{SFB2}, \cite[Cor 6.8]{SFB2}
Let $\bG$ be an affine group scheme of finite type over $k$, $r$ a positive integer, and 
 $M$ a finite dimensional $\bG_{(r)}$-module.  There is a natural morphism of commutative
 graded $k$-algebras \ $\psi: H^{ev}(\bG_{(r)},k) \ \to \ k[V_r(\bG)]$ of degree $\frac{p^r}{2}$ 
 with associated morphism  $\Psi: V_r(\bG) \ \to \ \Spec H^{ev}(\bG_{(r)},k)$ of schemes
 which is a homeomorphism on scheme-theoretic points with the Zariski topology.  
 This morphism restricts to a homeomorphism of closed subspaces
 $$\Psi: V_r(\bG)_M \quad \stackrel{\approx}{\to} \quad \Spec H^{ev}(\bG_{(r)},k)_M,$$
 where 
 $$V_r(\bG)_M \quad  = \quad \{ x \in V_r(\bG): \ J_{\bG,r,M}(x) \not= \frac{dim M}{p} \cdot [p]\}$$
 and $\Spec H^{ev}(\bG_{(r)},k)_M$ is the (affine) cohomological support variety of $M$.
  
Moreover, $\psi$ determines the ``projectivized morphism" \\ $\Psi: \Proj V_r(\bG) \ \to \ \Proj H^{ev}(\bG_{(r)},k)$ 
which restricts to a homeomorphism \\
 $\Psi: \Proj V_r(\bG)_ M \ \stackrel{\approx}{\to} \ \Proj H^{ev}(\bG_{(r)},k)_M.$
\end{thm}

\vskip .1in

\begin{remark}
\label{rem:Ngo}
We refer the reader to \cite{Ngo} for some explorations of $V_r(\bG)$, including the
reducibility and lack of equi-dimensionality in general, and some bounds on dimensions.
\end{remark}

\vskip .1in

\begin{remark}
The map $(\mu_x)_* \circ \epsilon_r$ of (\ref{eqn:lower*}) is a $\pi$-point of $\bG_{(r)}$ whose equivalence class
$[(\mu_x)_* \circ \epsilon_r]$ is a scheme-theoretic point of the $\pi$-point space $\Pi(\bG_{(r)})$.  
This determines the map of points induced by the 
homeomorphism 
\begin{equation}
\label{eqn:compose}
 \bP V_r(\bG) \quad \stackrel{\approx}{\to} \quad \bP H^*(\bG_{(r)},k) \quad \stackrel{\sim}{\to} \quad \Pi(\bG_{(r)})
 \end{equation}
 obtained by composing $\Psi$ of Theorem \ref{thm:fundamental}
with the natural isomorphism 
$\bP H^*(\bG_{(r)},k) \stackrel{\sim}{\to} \Pi(\bG_{(r)}) $ of \cite[Thm 7.5]{FP1}.

Observe that sending the 1-parameter subgroup $\mu_x: \bG_{a(r),k(x)} \to \bG_{k(x)}$ 
to the equivalence class of the $\pi$-point 
\ $(\mu_x)_*\circ \epsilon_r: k(x) \otimes k[u]/u^p \to  k(x) \otimes k\bG_{a(r)}  \to  k(x) \otimes k\bG$ \
determines a natural projection 
\begin{equation}
\label{eqn:proj}
V_r(\bG) {\text -} \{ 0 \} \quad \to \quad \bP V_r(\bG).
\end{equation}

For any finite dimensional $\bG_{(r)}$-module $M$, the maps of  (\ref{eqn:compose}) restrict to   
homeomorphisms of corresponding support varieties for $M$
$$
\bP V_r(\bG) _M\ \stackrel{\approx}{\to} \ \bP H^*(\bG_{(r)},k)_M \ \stackrel{\sim}{\to} \ \Pi(\bG_{(r)})_M.
$$
\end{remark}

\vskip .1in

\begin{remark}
\label{rem:pi-points}
The more general theory of $\pi$-points supports developed by Pevtsova and the author in \cite{FP1}, 
$M \to \Pi(H)_M$, applies to $H$-modules
for an arbitrary finite group scheme $H$.    A disadvantage of considering $M \to \Pi(H)_M$ rather than 
$M \mapsto V_r(\bG)_M$ is that points of $\Pi(H)$ are equivalence classes 
of certain flat maps of $K$-algebras of the form $\alpha_K: K[u]/u^p \to KH$ for field extensions $K/k$.
For maximal or generic points of $\Pi(H)$, the Jordan type of  a finite dimensional $H$-module $M$
is independent of the choice of representative of such an an equivalence class $[\alpha_K]$, but 
for other points of $\Pi(H)$ the Jordan type of $M$ depends upon a choice of 
representative of the equivalence class.  See \cite{FPS}).
\end{remark}
\vskip .1in

To emphasize the subtlety of the geometry of $\bP V_r(\bG)$, we analyze the very
special example of $\bP V_r(\bG_a)$.

\begin{ex}
\label{ex:weight}
We identify the graded coordinate algebra of $V_r(\bG_a) = V_r(\bG_{a(r)})$
with the graded polynomial algebra $k[T_0,T_1,\ldots,T_{r-1}]$, 
where $T_i$ is given degree $p^i$.  Then $\Proj$ applied to the graded algebra 
is the weighted projective space $w\bP(1,p,\ldots,p^{r-1})$.  Let $k[t_0,t_1,\ldots, t_{r-1}]$
denote the graded polynomial algebra with each $t_i$ of degree 1.
The graded map of $k$-algebras 
$$\phi^*: k[T_0,T_1,\ldots,T_{r-1}] \ \to \ k[t_0,t_1,\ldots, t_{r-1}], \quad T_i \mapsto t_i^{p^i}$$
determines a ramified covering map
$$\Phi:   \bP^{r-1} \quad \to \quad w\bP(1,p,\ldots,p^{r-1}).$$  
Similarly, the map of $k$-algebras 
$$\psi^*: k[t_0,t_1,\ldots,t_{r-1}] \ \to \ k[T_0,T_1,\ldots, T_{r-1}], \quad t_i \mapsto T_i^{p^{r-i-1}}$$
multiplies all degrees by $p^{r-1}$ and thus induces a morphism
$$\Psi: w\bP(1,p,\ldots,p^{r-1}) \quad \to \quad \bP^{r-1}.$$ 
Moreover,
$$\Psi \circ \Phi = F^{r-1}:\bP^{r-1} \ \to \bP^{r-1}, \quad \Phi \circ \Psi = F^{r-1}: 
w\bP(1,p,\ldots,p^{r-1}) \to w\bP(1,p,\ldots,p^{r-1}).$$

 This will enable us to compare Jordan type functions for $(\bG_a^{\times r})_1$
and $\bG_{a(r)}$  in Example \ref{ex:Ga}.
\end{ex}

\vskip .1in
\begin{ex}
\label{ex:sing}
Consider the case $r=3$, so that $\bP V_3(\bG_a)  \ = \ w\bP(1,p,p^2)$ is 2-dimensional,
covered by three affine open subsets:
$$U_0 \ = \ \Spec k[T_1/T_0^p,T_2/T_0^{p^2}] \ \simeq \ \bA^2,$$
$$U_1 \ = \ \Spec k[T_0^p/T_1,T_2/T_1^p] \ \simeq \ \bA^2,$$
$$U_2 \ = \ \Spec k[T_0^{ip}T_1^{p-i}/T_2; \ 0 \leq i \leq p].$$
Observe that $U_2 \ \subset \ w\bP(1,p,p^2)$ has an isolated singularity at $T_0 = 0 = T_1$;
the maximal ideal at this singular point is generated by the $p+1$ elements  \\
$T_0^{p^2}/T_2, T_0^{p^2-p}T_1/T_2,\ldots, T_0^pT_1^{p-1}/T_2,T_1^p/T_2$.
\end{ex}

\vskip .1in

The representability by $k[V_r(\bG)]$ of the functor sending a commutative $k$-algebra $A$ to 
the set of all height $r$ 1-parameter subgroups $\bG_{a(r),A} \to \bG_A$ (over $\Spec A$) leads to  
the universal 1-parameter subgroup $\cU_{\bG,r}$  as recalled in the following definition.

\begin{defn} (\cite[\S 2]{FP3})
\label{defn:univ}
Let $\bG$ be an affine group scheme of finite type
over $k$ and let $r$ be a chosen positive integer.  Then the universal 
1-parameter subgroup of height $r$ for $\bG$, 
\begin{equation}
\label{eqn:univG}
\cU_{\bG,r}:  \bG_{a(r),k[V_r(\bG)] } \ \to \bG_{(r),k[V_r(\bG)]},
\end{equation}
is the 1-parameter subgroup for the group scheme $\bG_{(r),k[V_r(\bG)]}$ over
$k[V_r(\bG)]$ corresponding to the identity map $V_r(\bG) \to V_r(\bG)$.
By the universality of $\cU_{\bG,r}$, for any scheme-theoretic point $x \in V_r(\bG)$
the restriction along $\chi_x: \Spec k(x) \to V_r(\bG)$ of $\cU_{\bG,r}$ equals
the map $\mu_x: \bG_{a(r),k(x)} \ \to \bG_{(r),k(x)}$ of Definition \ref{defn:VrG}.

We denote by 
\begin{equation}
\label{eqn:UGr}
(\cU_{\bG,r})_*: k[V_r(\bG)] \otimes k\bG_{a(r)} \quad \to \quad k[V_r(\bG)]\otimes k\bG_{(r)}
\end{equation}
the map of group algebras over $k[V_r(\bG)]$ induced by $\cU_{\bG,r}$. 
In other words, $(\cU_{\bG,r})_*$ is the $k[V_r(\bG)]$-dual of the map on coordinate algebras
$\cU_{\bG,r}^*: k[V_r(\bG)] \otimes k[\bG_{(r)}] \to k[V_r(\bG)] \otimes k[\bG_{a(r)}]$
given by $\cU_{\bG,r}$.
\end{defn}

\vskip .1in

By \cite[Prop 2.9]{FP3}, a closed embedding $i: \bH \hookrightarrow \bG$ of affine group schemes of 
finite type over $k$ determines the commutative square 
\begin{equation}
\label{eqn:restrict-square}
 \xymatrixcolsep{5pc}\xymatrix{
\bG_{a(r),k[V_r(\bG)]}  \ar[d] \ar[r]^{(\cU_{\bG,r})_*} &   \bG_{(r),k[V_r(\bG)]} \ar[d] \\
 \bG_{a(r),k[V_r(\bH)]} \ar[r]^{(\cU_{\bH,r})_*} &  \bH_{(r),k[V_r(\bH)] } ] .
}
\end{equation}

\vskip .1in

We recall and reformulate the ``universal $p$-nilpotent operator" $\Theta_{\bG,r}$ for $\bG_{(r)}$-modules 
introduced in \cite[Defn 2.1]{FP3}.

\begin{defn}
\label{defn:univ-Theta}
Let $\bG$ be an affine group scheme of finite type
over $k$ and let $r$ be a chosen positive integer.  
 We define 
\begin{equation}
\label{eqn:theta}
\Theta_{\bG,r} \quad  \equiv \quad (\cU_{\bG,r})_*(1\otimes u_{r-1}) \quad \in \quad 
k[V_r(\bG)]\otimes k\bG_{(r)}\ \equiv \ k\bG_{(r),k[V_r(\bG)]},
\end{equation}
where $(\cU_{\bG,r})_*$ is defined in (\ref{eqn:UGr}).  
\end{defn}

\vskip .in

In \cite[Ex 2.6]{FP3}, the reader will find a few explicit computations of $\Theta_{\bG,r}$.  We mention two
such examples consistent with the indexing used in this paper.

\begin{ex}
\label{ex:Theta}
Consider $\bG = \bG_a^{\times s}$ whose Lie algebra is the dimension $s$ Lie algebra $\fg_a^{\oplus s}$
(which is commutative, with $p$-th power map trivial).  Identify $V_1(\bG)$ with 
$\bA^r = \Spec k[t_0,\ldots,t_{s-1}]$
and identify the group algebra $k\bG_{(1)}$ with $k[x_0,\ldots,x_{s-1}]/(x_i^p) \\
= \ \fu(\fg_a^{\oplus s})$, the restricted enveloping algebra of the restricted Lie algebra $\fg_a^{\oplus s}$.
Then   
$$\Theta_{\bG,1} \ = \ \sum_{i=0}^{r-1} t_i \otimes x_i \ \in \ k[V_1(\bG_a] \otimes \fu(\fg_a^{\oplus s}).$$

Consider $\bG = \bG_a$ whose Lie algebra is $\fg_a$ and some $r > 0$.  
Identify $V_r(\bG_a)$ with $\Spec k[T_0,\ldots,T_{r-1}]$
with $T^i$ given grading $p^i$ and identify the group algebra $k\bG_{(r)}$ with $k[u_0,\ldots,u_{r-1}]/(u_i^p)$.  
Then
$$\Theta_{\bG.r} \ = \ \sum_{i=0}^{r-1} T_i^{p^{r-1-i}}\otimes u_i \ + \ \cdots 
\quad \in \ k[V_r(\bG_a)] \otimes k\bG_{a(r)},$$
where the non-explicit, remaining terms of $\Theta_{\bG_a,r}$ each consist of a
polynomial homogeneous of weighted degree  $p^{r-1}$ in the $T_i$'s  tensor a monomial in the $u_i$'s
of degree at least 2.  (See \cite[Eqn 6.5.1]{SFB2}.)
\end{ex}

\vskip .1in

 The somewhat subtle aspect of Theorem \ref{thm:theta}
is its last assertion of homogeneity of degree $p^{r-1}$ whose proof is given in \cite{FP3} (reflected in 
the examples above).

\begin{thm} (\cite{FP3})
\label{thm:theta}
Let $\bG$ be an affine group scheme of finite type
over $k$ and let $r$ be a chosen positive integer.  
Then $\Theta_{\bG,r} $ as given in Definition \ref{defn:univ-Theta} satisfies the following properties.
\begin{enumerate}
\item
For any scheme-theoretic point $x \in V_r(\bG)$, the specialization along $\chi_x: \Spec k(x) \to V_r(\bG)$ of 
$\Theta_{\bG,r} \in  k[V_r(\bG)]\otimes k\bG_{(r)}$ equals 
$(\mu_{x})_*(1\otimes u_{r-1}) \quad \in  \ k(x) \otimes k\bG_{(r)}$,
where $\mu_x$ is the 1-parameter subgroup of Definition \ref{defn:VrG}.
\item
Let $\phi: \bH \hookrightarrow \bG$ be a closed subgroup scheme determining
$\phi^*: k[V_r(\bG)] \twoheadrightarrow k[V_r(\bH)]$ and $\phi_{(r)*}: k\bH_{(r)} \hookrightarrow k\bG_{(r)}$.  
Then the restriction of $\Theta_{\bG,r}$ along $\phi^*$, 
$(\phi^* \otimes 1)(\Theta_{\bG,r}) \in k[V_r(\bH)] \otimes k\bG_{(r)}$, equals the image of 
 $(1\otimes \phi_{(r)*})(\Theta_{\bH,r})$.
\item
For any $\bG_{(r)}$-module $M$, $\Theta_{\bG,r}$
determines a $p$-nilpotent, $k[V_r(\bG)]$-linear operator
$\Theta_{\bG,r,M}: k[V_r(\bG)] \otimes M \ \to \ k[V_r(\bG)] \otimes M$.
\item
The base change along a scheme theoretic point $x \in V_r(\bG)$ of $\Theta_{\bG,r,M}$ equals the action of
$(\mu_x)_*(1\otimes u_{r-1})$ on $k(x) \otimes M$.  
\item
$\Theta_{\bG,r}$ is of the form $\sum f_i\otimes \phi_i \in k[V_r(\bG)]\otimes k\bG_{(r)}$ with each
$f_i$ homogeneous of degree $p^{r-1}$ in $k[V_r(\bG)]$ and each $\phi_i \in k\bG_{(r)}$
an element of the augmentation ideal $I_{\bG_{(r)}}$ of $k\bG_{(r)}$.
\end{enumerate}
\end{thm}

\begin{proof}
Since $\mu_x$ is the pull-back along $\chi_x$ of $\cU_{\bG,r}$, the following diagram commutes:
\begin{equation}
\label{eqn:pullback-theta}
 \xymatrixcolsep{5pc}\xymatrix{
k[V_r(\bG)] \otimes k\bG_{a(r)} \ar[d]^{\chi^*} \ar[r]^-{(\cU_{\bG,r})_*} & k[V_r(\bG)] \otimes k\bG_{(r)} \ar[d]^{\chi^*} \\
 k(x) \otimes k\bG_{a(r)} \ar[r]^-{(\mu_x)_*} &  k(x) \otimes k\bG_{(r)}. 
} 
\end{equation}
Consequently, the restriction of $\Theta_{\bG,r} = (\cU_{\bG,r})_*(1\otimes u_{r-1})$ equals
$(\mu_x)_*(1\otimes u_{r-1})$.

The naturality of $\bG \mapsto \Theta_{\bG,r}$ with respect to closed 
embeddings $\phi: \bH \hookrightarrow \bG$ as in assertion (2)
follows from the commutativity of (\ref{eqn:restrict-square}). 

The action by $\Theta_{\bG,r,M}$ is given by the 
restriction to $\{ \Theta_{\bG,r} \} \otimes (k[V_r(\bG)] \otimes M) $ of the tensor product 
pairing
\begin{equation}
\label{eqn:bilinear}
(k[V_r(\bG)] \otimes k\bG_{(r)}) \otimes (k[V_r(\bG)] \otimes M) \quad \to \quad k[V_r(\bG)] \otimes M.
\end{equation}
The fact that this action is $p$-nilpotent follows  from the observation that the $i$-th power of 
$\Theta_{\bG,r}$ equals $(\cU_{\bG,r})_*((1\otimes u_{r-1})^i)$.
Assertion (4) follows directly from assertion (1).

The last assertion is stated and proved in \cite[Prop 2.11]{FP3}.  The proof proceeds by first observing
that it suffices to verify that the restriction of the $k[V_r(\bG)]$-linear map 
$\cU_{\bG,r}^*:k[V_r(\bG)]\otimes k[\bG_{(r)}] \to 
k[V_r(\bG)]\otimes k[\bG_{a(r)}]$ to  $1\otimes k[\bG_{(r)}]$ reads off the coefficient of $t^{p^{r-1}}$
in $k[V_r(\bG)]\otimes k\bG_{a(r)} \simeq k[V_r(\bG)]\otimes k[t]/t^{p^r}$.  The authors then
verify that the actions of $\bA^1$ on $\bG_{a(r)}$ and $V_r(\bG)$ are appropriately compatible.
\end{proof}

\vskip .1in 

\begin{defn}
\label{defn:JT}
Let $\bG$ be an affine group scheme of finite type over $k$ and let $r$ be a
positive integer.  For any finite dimensional $\bG_{(r)}$-module $M$, we define the function 
\begin{equation}
\label{eqn:JT}
JT_{\bG,r,M}(-): V_r(\bG) \quad \to \quad \cY
\end{equation}
from the set of scheme-theoretic points of $V_r(\bG)$ to the set  of points of the partially orders
set $\cY$ by sending $x\in V_r(\bG)$ to
$$JT_{\bG,r,M}(x) \quad \equiv \quad JT((\mu_x)_* \circ \epsilon_r)^*(k[V_r(\bG)] \otimes M)),$$
the Jordan type of the $k(x)[u]/u^p$-module obtained by restricting $k[V_r(\bG)] \otimes M$ along
$(\mu_x)_* \circ \epsilon_r: k(x)[u]/u^p \to k(x)\otimes k\bG_{a(r)} \to k[V_r(\bG)] \otimes k\bG_{(r)}$;
equivalently, $JT_{\bG,r,M}(x)$ is the Jordan type of $(\mu_x)_*(1\otimes u_{r-1})$ acting on $k(x) \otimes M$.

By Theorem \ref{thm:theta}(1), $JT_{\bG,r,M}(x)$
equals the Jordan type of \ $k(x) \otimes_{k[V_r(\bG)]} \Theta_{\bG,r,M} \ \in \ 
End_{k(x)}(k(x) \otimes M)$.
\end{defn}

\vskip .1in

Let $R$ be a graded, commutative $k$-algebra; for example, $R \ = \ k[V_r(\bG)]$.
 For any graded $R$-module $M_*$, we denote by $(M_*)^{\sim}$ the associated quasi-coherent
sheaf on $\Proj (R)$.  We shall often consider graded $R$-modules of the form $R\otimes V$ where 
$V$ is a $k$-vector space and $R\otimes V$ inherits the grading of $R$.

If $M_*$ is a graded $R$-module, we denote by $M_*(n)$ the graded module
obtained from $M_*$ by ``shifting down by $n$" the degrees of $M_*$; thus, $(M_*(n))_d
= M_{n+d}$.  We use $(M_*)^\sim(n)$ to denote  $(M_*(n))^\sim$.
The operator $\Theta_{\bG,r,M}$ of Theorem \ref{thm:theta}(3) multiplies degrees by $p^{r-1}$, so that
viewed as an operator on graded modules one should write $\Theta_{\bG,r,M}$ as
$$\Theta_{\bG,r,M}: k[V_r(\bG)] \otimes M \quad \to \quad (k[V_r(\bG)] \otimes M)(p^{r-1}).$$

Because $k$ has positive characteristic,
the informative survey by M. Beltrametti and L. Robbiano in \cite{B-R} does not  apply to our 
consideration of $\bP V_r(\bG)$.  Because $V_r(\bG)$ is usually not
generated by homogeneous elements of degree 1, we supplement the clear exposition of 
R. Hartshorne in \cite[\S II.5]{Har} with \cite[Ex 9.5]{Eisen} when considering the relationship 
between graded modules over a commutative,  graded $k$-algebra $R$ and sheaves of 
modules on $\Proj (R)$.  

\vskip .1in

\begin{prop}
\label{prop:weighted}
Consider a polynomial algebra $k[T_0,\ldots,T_N]$ with each $T_i$ given a  ``weighted" degree  $d_i$.   
  Assume that
$I_Y \subset k[T_0,\ldots,T_N]$ is an ideal generated by polynomials $F_j(\ul T)$ 
which are (weighted) homogeneous.   Then, the zero locus $Y = Z(I_Y) \ \subset 
\Proj k[T_0,\ldots,T_N]$ is a weighted projective variety associated to the 
commutative, graded algebra $k[T_0,\ldots,T_N]/I_Y$.

Let $d$ be a positive integer divisible by each $d_i$.
The graded $k[T_0,\ldots,T_N]/I_Y$-module $(k[T_0,\ldots,T_N]/I_Y)(d)$ 
is a locally free, coherent rank 1 sheaf of $\cO_Y$-modules on $Y$.  
We denote this sheaf by $\cO_Y(d)$.  For each $i$ such that
$T_i$ is not in the radical of $I_Y$, the restriction $(\cO_Y(d))_{|V_i}$ of $\cO_Y(d)$ to open subvariety
$$V_i \ \equiv (\Proj k[T_0,\ldots,T_N]) - Z(T_i)) \cap Y \ \subset \Proj k[T_0,\ldots,T_N]$$ 
is generated as  a free, rank 1 $\cO_{V_i}$-module by the image of $T_i^{d/d_i} \in \cO_Y(d)(Y)$ in 
$\cO_Y(d)(V_i)$.
\end{prop}

\begin{proof}
It suffices to assume that $I_Y$ is a radical ideal .  Moreover, by considering each irreducible component
of $Y$, we may assume that $Y$ is irreducible.  Then  \cite[Ex 9.5]{Eisen} applies to verify that
 the variety $Z(I_Y)$ is isomorphic to its image in $\Proj S$,
where $S \ = k[T_0,\ldots,T_N]_{(d)} \subset k[T_0,\ldots,T_N]$ is the graded subalgebra generated by
(weighted) homogeneous polynomials each of whose weights is divisible by $d$ and 
the image of $Z(Y)$ is given by the homogeneous ideal $I_Y\cap S \subset S$ which is generated
by elements of degree $d$.  Thus, replacing $k[T_0,\ldots,T_N]$
by $S$ with grading modified by dividing degrees by $d$, we may apply \cite[Prop II.5.12]{Har}
(which is stated for graded submodules of a commutative, graded algebra generated in degree 1).
\end{proof}

In particular, Proposition \ref{prop:weighted} applies to the graded commutative algebra $k[V_r(\bG)]$
which is generated by elements homogeneous of degree $p^i, 1\leq i < r$.  (See Proposition \ref{prop:grading}.)

\vskip .1in

We next present the projectivization of the universal operator  $\Theta_{\bG,r}$.  This operator
has been used in \cite{FP3} to produce global invariants on $\bP V_r(\bG)$  for $\bG$-modules (especially, 
non-trivial vector bundles on $\bP V_r(\bG)$) in contrast to the Jordan type function
$JT_{\bG,r,M}(-)$ which is a ``point-wise invariant" on $\bP V_r(\bG)$ invariant for a $\bG$-module $M$.
See Section \ref{sec:bundles} below.

\vskip .1in

\begin{thm}
\label{thm:global-univ}
Let $\bG$ be an affine group scheme of finite type
over $k$ and let $r$ be a chosen positive integer.  
The universal operator  $\Theta_{\bG,r}$ of Theorem \ref{thm:theta} determines a global section 
$$\tilde  \Theta_{\bG,r} \quad \in \quad  \Gamma(\bP V_r(\bG),\cO_{\bP V_r(\bG)}(p^{r-1})\otimes k\bG_{(r)})$$
of the coherent, locally free sheaf $\cO_{\bP V_r(\bG)}(p^{r-1})\otimes k\bG_{(r)})$ on
$\bP V_r(\bG)$ satisfying the following properties.
\begin{enumerate}
\item
For any $\bG_{(r)}$-module $M$, tensor product with $\tilde  \Theta_{\bG,r}$ determines the 
map of $\cO_{\bP V_r(\bG)}$-modules
$$\tilde \Theta_{\bG,r,M}: (k[V_r(\bG)]\otimes M)^{\sim} \quad \to \quad (k[V_r(\bG)]\otimes M)^\sim(p^{r-1}).$$
\item
For any non-zero divisior $h \in  k[V_r(\bG)]_{p^{r-1}}$, let $U_h$ denote the affine open subset \
$Spec (k[V_r(\bG)][h^{-1}])_0 \ \subset \bP V_r(\bG)$.  The natural  isomorphism of
 $(\cO_{\bP V_r(\bG)})_{U_h}$-modules
$$h^{-1} \otimes 1: (\cO_{\bP V_r(\bG)}(p^{r-1}) \otimes k\bG_{(r)})_{|U_h} 
\quad \stackrel{\sim}{\to} \quad  (\cO_{\bP V_r(\bG)} \otimes k\bG_{(r)})_{|U_h}$$
sends 
$$(\tilde \Theta_{\bG,r})_{|U_h} \ \in \ \Gamma(U_h,\cO_{\bP V_r(\bG)}(p^{r-1}) \otimes k\bG_{(r)}) \ = \
(k[V_r(\bG)][1/h])_{p^{r-1}} \otimes k\bG_{(r)} $$
to 
$$(h^{-1}\otimes 1)\cdot\Theta_{\bG,r} \in k[V_r(\bG)][1/h]\otimes k\bG_{(r)}.$$
\item
 If $\ol x$ is a scheme-theoretic point of $\bP V_r(\bG)$ satisfying the condition that $h(\ol x) \not= 0$
and if $x$ is a scheme theoretic point of $V_r(\bG) {\text -} \{ 0 \}$ projecting to $\ol x$, 
then the restriction 
$k(x) \otimes_{\cO_{\bP V_r(\bG)}} \tilde\Theta_{\bG,r} $ along 
$\chi_{\ol x}:\Spec k(x) \to \bP V_r(\bG)$ of $\tilde  \Theta_{\bG,r}$
is given by 
\begin{equation}
\label{eqn:restrict-tildeTheta}
k(x) \otimes_{\cO_{\bP V_r(\bG)}} \tilde\Theta_{\bG,r} \quad = \quad  
c \cdot (\mu_x)_*(1 \otimes u_{r-1}) \in  k(x) \otimes  k\bG_{(r)}
\end{equation}
for some $0 \not= c \in k(x)$.
\item
For any $\bG_{(r)}$-module $M$, the restriction of $\tilde \Theta_{\bG,r,M}$ along 
$\chi_{\ol x}: \Spec k(x) \to \bP V_r(\bG)$
is a non-zero scalar multiple of the image of $(\mu_x)_*(1\otimes u_{r-1})$ in \\
$End_{k(x)}(k(x)\otimes M$).
\end{enumerate}
\end{thm}

\begin{proof}
The global sections of the sheaf $\cO_{\bP V_r(\bG)}(p^{r-1})\otimes k\bG_{(r)}$ on $\bP V_r(\bG)$
are those elements of degree 0 of the graded $k[V_r(\bG)]$-module $k[V_r(\bG)](p^{r-1})\otimes k\bG_{(r)}$.
Since $ \Theta_{\bG,r} \in k[V_r(\bG)]\otimes k\bG_{(r)}$ is homogeneous 
of degree $p^{r-1}$, it determines the global section $\Theta_{\bG,r}$ of $\cO_{\bP V_r(\bG)}(p^{r-1})\otimes k\bG_{(r)}$.

The first assertion of the theorem follows from Theorem \ref{thm:theta}(3) and the correspondence
between graded $k[V_r(\bG)]$-modules and quasi-coherent sheaves on $ \bP V_r(\bG)$.
The second assertion follows from the standard description of principal open subsets of
  $\Proj R$ associated to homogenous elements $h \in R$ as $\Spec R[h^{-1}]$.

The equality (\ref{eqn:restrict-tildeTheta}) follows from Theorem \ref{thm:theta}(1) and the identification
of assertion (2).  Assertion (4) follows immediately from assertion (3) (see Theorem \ref{thm:theta}(4)).
\end{proof}

\vskip .1in

\begin{remark}
For any $\bG_{(r)}$-module $M$ of dimension $m$ with $k$-linear dual $M^*$, 
	$$JT_{\bG,r,M}(-) \quad = \quad  JT_{\bG,r,M^*}(-):\bP V_r(\bG) \to \cY_{m}.$$
Moreover,  if $I$ and $J$ restrict to 
injective $\bG_{(r)}$ modules of the same dimension, then  $JT_{\bG,r,I}(-) = JT_{\bG,r,J}(-)$.  Thus,
$M \mapsto JT_{\bG,r,M}(-)$ is far from faithful.
\end{remark}

\vskip .1in

The next proposition verifies that $JT_{\bG,r,M}$ is a ``function on $\bG$-conjugacy classes".
Fix the convention that the action $(-)^g: G \to G$  for a group $G$ with $g \in G$ sends $h \in G$
to $g^{-1}\cdot h \cdot g \in G$.  This easily leads to the adjoint action of $\bG$ on itself
(and on its normal subgroups $\bG_{(r)}$).  Composition of a 1-parameter group $\bG_{a(r)} \to \bG$
with this adjoint action determines an ``adjoint action" of $\bG$ on $V_r(\bG)$.

\begin{prop}
\label{prop:adjoint}
Consider an affine group scheme $\bG$ of finite type over $k$, $r$ a positive integer,
$M$ a $\bG$-module of dimension $m$, and 
$\chi_x: \Spec k(x) \to V_r(\bG)$ a scheme-theoretic point of $V_r(\bG)$
corresponding to the 1-parameter subgroup $\mu_x: \bG_{a(r),k(x)} \to \bG_{k(x)}$.
For $g \in \bG(k(x))$, denote by 
$\chi_{x^g}: \Spec k(x) \to V_r(\bG)$ the 
the scheme-theoretic point corresponding to the 1-parameter subgroup
 $(-)^g \circ \mu_x:  \bG_{a(r),k(x)} \to  \bG_{k(x)}$.
Then 
$$JT_{\bG,r,M}(x^g) \quad = \quad JT_{\bG,r,M}(x).$$
\end{prop}

\begin{proof}
Using base change from $k$ to $k(x)$, we reduce to the notationally simpler case that $k = k(x)$.
Let $M^g$ denote the $\bG$-module with
the same underlying vector space structure as that of $M$ and with $h \in \bG(k)$ sending
 $n \in M^g$ to $(g^{-1}hg)\cdot n \in M^g$ where $\cdot$ refers to the given action of $\bG(k)$ on $M$.
 Observe that the map $g: M^g \to M, \ n \mapsto g\cdot n$ is $\bG(k)$-equivariant.
 
 This leads to a commutative diagram (viewing group schemes as functors from commutative $k$-algebras
 to groups)
 \begin{equation}
\label{eqn:cong}
 \xymatrixcolsep{5pc}\xymatrix{
k[V_r(\bG)] \otimes M^g  \ar[r]^{1\otimes g\bu} \ar[d]^-{\Theta_{\bG,r,M^g}} & k[V_r(\bG)] \otimes M \ar[d]^-{\Theta_{\bG,r,M}}  \\
k[V_r(\bG)] \otimes M^g \ar[r]^{1\otimes g\bu} \ar[d]^{(\mu_{x^g})^*\otimes 1} & 
k[V_r(\bG)] \otimes M \ar[d]^{(\mu_x)^*\otimes 1}  \\
k[V_r(\bG_a)] \otimes M^g  \ar[r]^{1\otimes g\bu} &   k[V_r(\bG_a)] \otimes M .
}
\end{equation}
Consequently,  the action of $(\mu_{x^g})_*(1\otimes u_{r-1}) = (\Theta_{\bG,r,x^g})_*(1\otimes u_{r-1})$ on $M^g$
is isomorphic to the action of $(\mu_x)_*(1\otimes u_{r-1}) = (\Theta_{\bG,r,x})_*(1\otimes u_{r-1})$ on $M$,
where $\Theta_{\bG,r,x^g}$ and  $\Theta_{\bG,r,x}$ are the specializations of $\Theta_{\bG,r}$ at $x^g$ and $x$.
\end{proof}

\vskip .1in

 In the theorem below, we equip (the scheme-theoretic points of) $\bP V_r(\bG)$ with the Zariski topology 
and equip $\cY_{m}$ with the Alexandrov 
discrete topology (see Theorem \ref{thm:cont}).  We refer to \cite[Prop 2.8.3]{FP2} for an earlier 
discussion of properties of the Jordan type function.

\begin{thm}
\label{thm:JT}
Let $\bG$ be an affine group scheme of finite type over $k$ and let $r$ be a
positive integer.   
\begin{enumerate}
\item
For any finite dimensional $\bG_{(r)}$-module $M$, the function $JT_{\bG,r,M}(-)$ of (\ref{eqn:JT})
induces a continuous map
$$JT_{\bG,r,M}(-): \bP V_r(\bG) \quad \to \quad \cY$$ 
sending a scheme-theoretic point $\ol x \in \bP V_r(\bG)$ to $JT_{\bG,r,M}(x)$
for any choice of scheme-theoretic point $x \in V_r(\bG)$ projecting to $\ol x$.
\item
The finite dimensional $\bG_{(r)}$-module $M$ is injective if and only if the dimension $m$ of $M$ is divisible
by $p$ and \ $JT_{\bG,r,M}(x) \ = \ \frac{m}{p} \cdot [p]$ \ for all $x \in \bP V_r(\bG)$.
\item
For any $\ol x \in \bP V_r(\bG)$, $M \in mod(\bG_{(r)},m)$, and $N \in mod(\bG_{(r)},n)$, 
$$JT_{\bG,r,M\oplus N}(\ol x) \ = \ JT_{\bG,r,M}(\ol x)+JT_{\bG,r,N}(\ol x) \  \in \ \cY_{m+n}.$$
\item
If $r=1$ or if $\ol x \in \bP V_r(\bG)$ is a generic point, if $M \in mod(\bG_{(r)},m)$, and if 
$N \in mod(\bG_{(r)},n)$, then
$$JT_{\bG,r,M\otimes N}(\ol x) \ = \ JT_{\bG,r,M}(\ol x)\otimes JT_{\bG,r,N}(\ol x) \  \in \ \cY_{m\cdot n}.$$
\item
If $f: M \to N$ is a surjective map of 
finite dimensional $\bG_{(r)}$-modules, then $JT_{\bG,r,N}(\ol x) \ \leq \ JT_{\bG,r,M}(\ol x)$
for any $\ol x \in \bP V_r(\bG)$;  in other words, $JT_{\bG,r,N}(-) \leq JT_{\bG,r,M}(-)$, where the partial 
ordering of functions is that inherited by the partial ordering of $\cY$.
\item
Consider a short exact sequence $\cE: 0 \to N \ \to \ M \ \to \ Q \to 0$ of finite dimensional $\bG_{(r)}$-modules.
Then \ $JT_{\bG,r,M}(-) =  JT_{\bG,r,N}(-) + JT_{\bG,r,Q}(-)$ \
if and only if the short exact sequence $\mu_x^*(\cE)$ of $\bG_{k(x)}$-modules splits for all $x \in V_r(\bG)$.
\item
If  $\phi: \bH \to \bG$ is a morphism of affine group schemes of finite type over $k$, then for any
finite dimensional $\bG$-module $M$
$$JT_{\bH,r,\phi^*(M)}(-): \bP V_r(\bH) \ \to \ \cY$$ equals
$$JT_{\bG,r,M}(-) \circ \phi_*: \bP V_r(\bH) \ \to \ \bP V_r(\bG) \ \to \ \cY.$$
\end{enumerate}
\end{thm}

\begin{proof}
The fact that $JT_{\bG,r,M}(\ol x)$ is well defined (namely, independent of the choice of $x \in V_r(\bG)$
projecting to $\ol x \in \bP V_r(\bG)$),
follows from the equality (\ref{eqn:restrict-tildeTheta}) of Theorem \ref{thm:global-univ}.  
To prove the asserted continuity of assertion (1) it suffices to prove continuity of
$JT_{\bG,r,M}(-)$ when restricted to each affine open subset of the form
$\Spec (k[V_r(\bG)][1/h])_0  \subset \bP V_r(\bG)$,
  where $h \in k[V_r(\bG)]_{p^{r-1}}$ is an irreducible  homogeneous generator of $k[V_r(\bG)]$
raised to the power $p^{r-1}/deg(h)$.
This follows from the continuity proved in Theorem \ref{thm:cont} for $A = (k[V_r(\bG)][1/h])_0$.

Assertion (2)
follows from the detection theorem for nilpotence for $\bG_{(r)}$-modules given in \cite[Thm 4.3, Cor 4.4]{SFB2}.

The fact that $JT_{\bG,r,M\oplus N}(-) = JT_{\bG,r,M}(-) + JT_{\bG,r,N}(-)$ is self evident.
Since $(\mu_x)_*: k(x) \otimes k\bG_{a(r)} \to k(x) \otimes k\bG_{(r)}$ is a map of 
Hopf algebras (over $k(x)$), the restriction map $\mu_x^*: mod(k(x) \otimes k\bG_{(r)}) 
\to mod(k(x) \otimes k\bG_{a(r)})$ commutes with tensor products.  
Since $\epsilon_1$ is the identity, this implies assertion (4) whenever $r=1$.
On the other hand, if $x$ is a generic point of $V_r(\bG)$ for any $r > 0$, 
the fact that $M \ \mapsto \ JT_{\bG,r,M}(x)$ commutes with tensor products  is given 
by  \cite[Prop 4.7]{FPS}.

A surjective map $f: M \to N$ of finite dimensional $\bG_{(r)}$-modules induces a 
surjective map $k(x)\otimes M \ \to \ k(x) \otimes N$ of $k(x)[u]/t^p$-modules.
Consequently, assertion (5) follows from Corollary \ref{cor:surjective}. 
Assertion (6) follows from the definition of a locally split exact sequence (see \cite{CF})
and the fact for any field $K$ and any short exact sequence $0 \to N_K \to M_K \to Q_K \to 0$
that the Jordan type of $M_K$ equals the sum of the Jordan types of $N_K$ and $Q_K$ if and
only if the short exact sequence splits.

Assertion (7) is equivalent to the assertion that 
$JT_{\bG,r,M}(\phi_*(y)) = JT_{\bH,r,\phi^*(M)}(y)$ for every $y \in \bP V_r(\bH)$.  This
 is immediate from the identification (with $x = \phi_*(y)$) \\
$JT((\mu_x)_* \circ \epsilon_r)^*(k(x) \otimes M) \quad = 
\quad JT((\mu_y)_* \circ \epsilon_r)^*(k(y) \otimes \phi^*(M)).$
\end{proof}

\vskip .1in

The following corollary is rephrasing of some of the properties of Theorem \ref{thm:JT}.

\begin{cor}
\label{cor:Hom-cont}
We denote by $Hom_{cont}(\bP V_r(\bG),\cY)$ the set of continuous
maps from $\bP V_r(\bG)$ with the Zariski topology to $\cY$  with the Alexandrov 
discrete topology.  Give (the scheme-theoretic points of ) $\bP V_r(\bG)$ the partial 
ordering of specialization and give $Hom_{cont}(\bP V_r(\bG),\cY)$ the partial 
ordering inherited from that of $\cY$.

Then sending a finite dimension $\bG_{(r)}$-module $M$ to $JT_{\bG,r,M}(-)$ determines 
$$JT_{\bG,r}: mod(\bG_{(r)}) \quad \to \quad Hom_{cont}(\bP V_r(\bG),\cY)$$
commuting with sums and the partial orderings.
\end{cor}

\vskip .1in

The following theorem associates to a finite dimensional $\bG$-module $M$
the configuration $ \{ (\bP V_r(\bG))_{M, \leq \ul a} \}$ which gives considerably more 
 information about $M$ than does the support variety $\bP V_r(\bG)_M \subset \bP V_r(\bG)$
of $M$.  The fact that 
$(V_r(\bG))_{M, \leq \ul a} \subset V_r(\bG)$ is closed is given with a sketch of proof in \cite[Prop 2.8.3]{FP2}.

\begin{thm} (see Proposition 2.8.3 of \cite{FP2})
\label{thm:gen-strata}
Let $\bG$ be an affine group scheme of finite type over $k$ and let $r$ be a
positive integer.   Consider a $\bG_{(r)}$-module $M$ of dimension $m$.
For each $\ul a \in \cY_{m}$, we define the subset
$$(\bP V_r(\bG))_{M, \leq \ul a} \ \equiv \ \{\ol x:  JT_{\bG,r,M}(\ol x) \leq \ul a\} \ \subset \
\bP V_r(\bG).$$
Then $\{ (\bP V_r(\bG))_{M,\leq \ul a }; \ \ul a \in \cY_{m} \}$ is a 
partially ordered configuration of closed subspaces of $\bP V_r(\bG)$.

With this notation, the support variety $\bP V_r(\bG)_M$ of \cite{SFB2} equals the union of those 
subspaces $(\bP V_r(\bG))_{M, \ul a \leq}$ indexed by $\ul a \in \cY_{m}$ which 
satisfy the condition that $a_i \not= 0$ for some $ i< p$.  (See Theorem \ref{thm:fundamental}.)

The non-maximal support variety $\Gamma(\bG_{(r)})_M$ of \cite{FP2} equals the union
of those subspaces $(\bP V_r(\bG))_{M, \ul a \leq}$ indexed by those $\ul a \in \cY_{m}$ which 
satisfy the condition that there exists some $\ul b \in \cY_{m}$ with  $\ul a \leq \ul b$ 
such that $(\bP V_r(\bG))_{M, \leq \ul a }$ is strictly contained in $(\bP V_r(\bG))_{M, \leq \ul b}$
\end{thm}

\begin{proof}
The continuity of $JT_{\bG,r,M}(-)$ together with the evident fact for any $\ul a \in \cY$ that 
$\{ \ul c \in \cY: \ \ul c  \leq a \}$ is closed in $\cY$ implies 
$(\bP V_r(\bG))_{M, \leq \ul a}$ is closed in $\bP V_r(\bG)$.
The identification of the subspaces $\bP V_r(\bG)_M =
 \cup_{\ul a \not= n[p]} (\bP V_r(\bG)_{M, \leq \ul a})$
 and $\Gamma(\bG_{(r)})_M$ (as asserted above) inside $\bP V_r(\bG)$
 follows immediately from the definitions of these subspaces.
\end{proof}

\vskip .1in

We make explicit the following corollary first stated and proved in \cite[Cor 2.15]{FP2}.

\begin{cor}
\label{cor:locally closed}
Let $M$ be a $\bG_{(r)}$-module of dimension $m$, and let $\ul a$ be an object on $\cY_{\leq,m}$, The subset 
$$(\bP V_r(\bG))_{M, =  \ul a} \ \equiv \ \{ x:  JT_{\bG,r,M}(x) = \ul a\} \ \subset \ \bP V_r(\bG)$$
is locally closed, the union of those $(\bP V_r(\bG))_{M, \leq \ul b}$ such that 
$\ul b \leq \ul a$ and $\ul b \not= \ul a$,
$$(\bP V_r(\bG))_{M, =  \ul a} \quad = \quad (\bP V_r(\bG))_{M, \leq \ul a} - 
(\cup_{\ul a \not= \ul b \leq \ul a}(\bP V_r(\bG))_{M, \leq \ul b}).$$
\end{cor}

\vskip .1in

\begin{remark}
\label{rem:colored}
We can interpret Corollary \ref{cor:locally closed} as asserted that the Jordan type 
function $JT_{\bG,r,M}(-)$ for a finite dimensional $\bG_{(r)}$-module $M$ determines
a stratification of $\bP V_r(\bG)$ with locally closed strata, each colored by a 
Young diagram such that if one stratum is strictly contained in the closure of another stratum
then the Young diagram of the smaller stratum is less than the Young diagram of 
the bigger stratum.
\end{remark}

\vskip .1in

\begin{remark}
\label{rem:generalized}
In \cite[Defn 4.1]{FP2}, generalized support varieties $\Gamma^j(\bG)_M$ were introduced
for finite dimensional $\bG$-modules $M$ for an arbitrary finite group scheme $\bG$ over $k$.  
Namely, \ $\Gamma^j(\bG)_M$ is the subspace of those $\pi$-points $ [\alpha_K] \in \Pi(\bG)$
with the property that the rank of $\alpha_K(t^j)$ on $M_K$ is not maximal.
For $\bG$
infinitesimal, one can recover $\Gamma^j(\bG)_M$ as in Theorem \ref{thm:gen-strata} by 
replacing $\Theta_{\bG,r,M}$  in Theorem \ref{thm:theta} by its $j$-th power,
then replacing $JT((\mu_x)_* \circ \epsilon_r)^*(k[V_r(\bG)] \otimes M)$ in Theorem \ref{thm:JT}
by the Jordan type of the $j$-th power of $(\mu_x)_*(u_{r-1})$ acting on $k(x)\otimes M$.
\end{remark}

\vskip .2in


\section{First examples}
\label{sec:1examples}

In this section, we consider a few general examples of computations of $JT_{\bG,r,M}(-)$ 
(some found in \cite[\S 2]{FP2}).   More specific computations for the simpler $JT^{exp}_{\bG,r,M}(-)$
are given in Section \ref{sec:apps}.

\vskip .1in

\begin{ex}
\label{ex:constant-type}
Let $\bG = \bG_{(r)}$ be an  infinitesimal group scheme of height $\leq r$.  A $\bG$-module $M$ of 
dimension $m$ is of constant Jordan type 
$\ul a$ if and only if  $JT_{\bG,M,r}(-)$ is the constant function
sending each $\ol x \in \bP V_r(\bG)$ to $\ul a \in \cY_{m}$.  (See \cite{CFP}
for numerous examples.)

In particular, if $I$ is an injective $\bG$-module of dimension $m$,  then $p$ divides $m$
and $JT_{\bG,I,r}(-)$ sends each $\ol x \in \bP V_r(\bG)$ to $\frac{m}{p}[p] \in \cY_{m}$
(i.e., $JT_{\bG,I-r}(\ol x) = \frac{m}{p}[p] + 0[p-1]+ \cdots + 0[1]$).
\end{ex}

\vskip .1in

\begin{ex}
\label{ex:Carlson-mod}
Let $\bG = \bG_{(r)}$ be an infinitesimal group scheme of height $\leq r$ and $\ol x \in \bP V_r(\bG)$.
Consider $0 \not= \zeta \in H^{2d}(\bG,k)$ and take $M$ to be the Carlson 
$L_\zeta$-module as first constructed in \cite{C}; namely, $M$ is the kernel of a map
$\tilde \zeta:\Omega^{2d}(k) \to k$ representing $\zeta$.   

Then the dimension $n$ of $M$ is divisible by $p$.  Moreover, 
$$J_{\bG,M,r}(\ol x) = \frac{n}{p}[p] \quad {\text if}  \quad ((\mu_x)_* \circ \epsilon_r)^*(\zeta)) \not= 0
 \in H^{2d}(k(x)[u]/u^p,k),$$
$$J_{\bG,M,r}(\ol x) = (\frac{n}{p}-1)[p]+[p-1]+[1] \quad {\text if} \quad  ((\mu_x)_* \circ \epsilon_r)^*(\zeta)) = 0
 \in H^{2d}(k(x)[u]/u^p,k).$$
\end{ex}

\vskip .1in

\begin{ex}  (See \cite[Prop 2.9]{FP2})
\label{ex:perp}
Let $\bG = \bG_{(r)}$ be an infinitesimal group scheme of height $\leq r$.
Let $M$ be a finite dimensional $\bG$-module and $P \twoheadrightarrow M$ a finite dimensional
projective cover of $M$ with kernel $\Omega(M)$.   If 
$JT_{\bG,M,r}(\ol x) \ = \ul a$ for some $\ol x \in \bP V_r(\bG)$, then
$JT_{\bG,\Omega(M),r}(\ol x) \ = \ul a^\perp + n[p]$ for some integer $n$,
where $\ul a^\perp = \sum_{i=1}^{p-1}a_{p-i}[i]$.
\end{ex}

\vskip .1in

\begin{ex}
\label{ex:product-special}
(See \cite[\S I.3.8]{J}.)
Consider $\bG_1, \ \bG_2$, infinitesimal group schemes of height $\leq r$.
  Let $\bG = \bG_1 \times \bG_2$ and $M_1$ be a
a finite dimensional $\bG_1$-module.  Set $M$ equal to $ind_{\bG_1}^{\bG}M_1 \simeq M_1 \otimes k[\bG_2]$.
Let $\ol {(x,y)}$ be a scheme-theoretic point in $\bP V_r(\bG)$ represented by the scheme-theoretic
point $(x,y) \in V_r(\bG_1) \times V_r(\bG_2)$.
Then 
$$JT_{\bG,M,r}(\ol  {(x,0}))  \ = \ dim(k[\bG_2])\cdot JT_{\bG_1,M_1}((\ol x)), $$
$$JT_{\bG,M,r}(\ol  {(x,y})) \ = \ dim(M)/p\cdot [p], \quad y \not=0.$$
\end{ex}

\vskip .1in

We generalize the previous example as follows.

\begin{ex}
\label{ex:product}
Consider $\bG_1, \ \bG_2$, infinitesimal group schemes of height $\leq r$, a finite dimensional
$\bG_1$-module $M_1$, and a finite dimensional $\bG_2$-module $M_2$.
Let $(x,y)$ be a scheme-theoretic point in $V_r(\bG_1) \times V_r(\bG_2)$.   Then
$$JT_{\bG_1\times \bG_2,M\otimes N,r}((x,y)) \quad = \quad JT_{\bG_1,M,r}(x) \otimes JT_{\bG_2,N,r}(y).$$
\end{ex}

\vskip .1in 

\begin{ex}
\label{ex:ht1}
Let $\bG$ be an infinitesimal group scheme of height 1 with Lie algebra $\fg$.  Thus, 
$k\bG$ is naturally identified with $\fu(\fg)$, the restricted enveloping algebra of $\fg$
and $V_1(\bG)$ is naturally identified with the  $p$-nilpotent cone $\cN_p(\fg) \subset \fg$.
(See \cite[Lem 1.6]{SFB1}.)

Let $V$ be a faithful representation of $\bG$ of dimension $N$, and associate to $V$
a closed embedding $\phi: \bG \hookrightarrow GL_{N(1)} \hookrightarrow GL_N$.
For any $\alpha \in \cN_p(\fg)(A)$, the associated 1-parameter subgroup of 
$\mu_\alpha: \bG_{a(1),A} \to \bG_A$ sends
$s\in \bG_{a(1)}(A)$ (i.e., $s$ a $p$-nilpotent element of $A$) to 
$$\mu_\alpha(s) \ = 1 + s\cdot \alpha + \frac{s^2}{2}\alpha^2 +
\cdots + \frac{s^{p-1}}{(p-1)!}\alpha^{p-1} \ \in \ \bG_{(1)}(A) \ \subset \ GL_N(A),$$
where the powers of $\alpha$ are taken in $\fg l_N$.
We associate to $\mu_\alpha$ the linear operator $(\mu_{\alpha})_*: A\otimes k\bG_{a(1)} \to \A\otimes k\bG$.

Let $M$ be a finite dimensional $\bG$-module of dimension $m$ and choose a scheme-theoretic
point $\alpha \in \cN_p(\fg)$, given by $\alpha \in k(\alpha)\otimes \cN_p(\fg)$.  Then
$JT_{\bG,M,1}(\alpha)$ is the Jordan type of $(\mu_{\alpha})_*(1\otimes u) \in k(\alpha)\otimes k\bG$
acting on $k(\alpha)\otimes M$.   This action is the action of $\alpha$ as an element of the 
Lie algebra $k(\alpha)\otimes \fg$ action on $k(\alpha)\otimes M$. 
\end{ex}

\vskip .1in

\begin{ex}
\label{ex:Gar} 
Assume $\bG = \bG_{a(r)}$ for some $r > 0$.  Then $\bG$ admits an embedding
in $GL_2$ as the $r$-th Frobenius kernel of the unipotent subgroup $U_2$ of strictly upper 
triangular matrices.  As in Example \ref{ex:weight},
$\bP V_r(\bG)$ is a weighted projective space with weights $(1,p,\cdots,p^{r-1})$
associated to the graded polynomial algebra $k[T_0,T_1,\ldots,T_{r-1}]$ where $T_i$ has 
weight $p^i$. 
Consider a non-zero $r$-tuple $\ul a = (a_0,\ldots,a_{r-1})$ of elements of $k(x)$ determining a 
scheme-theoretic point $x \in V_r(X)$ projecting to the scheme-theoretic point 
$\ol x \in \bP V_r(\bG)$, and consider the associated 1-parameter
subgroup  $\mu_x:  \bG_{a(r),k(x)} \to \bG_{a(r),k(x)}.$ 
An explicit (but quite complicated) formula for $\mu_{x*}(1\otimes u_{r-1}) \in k(x) \otimes k\bG_{a(r)}
\ = \ k(x)\otimes k[u_0,\ldots,u_{r-1}]/(u_i^p)$ 
 is given in \cite[(6.5.12)]{SFB2}.  

The sum of terms linear in the $u_i$'s in the 
expression for $\mu_{x*}(1\otimes u_{r-1})$  is $\sum_{i=0}^{r-1} a_{r-1-i}^{p^i}u_i\in k(x) \otimes k\bG$.
 Replacing $\mu_{x*}(1\otimes u_{r-1})$ by $\sum_{i=0}^{r-1} a_{r-1-i}^{p^i}u_i$
is a special case of replacing $\Theta_{\bG,r}$ by $\Theta^{exp}_{\fg,r}$ as discussed in Section \ref{sec:exp}.
\end{ex}

\vskip .2in


\section{Exponential Jordan type}
\label{sec:exp}

	In this section, we consider affine group schemes whose 1-parameter subgroups 
are given by exponentiating $r$-tuples of $p$-nilpotent, pair-wise commuting elements
of $\fg \equiv Lie(\bG)$.  The more explicit nature of these 1-parameter subgroups 
allows us to simplify our Jordan type functions, not necessarily preserving their values 
but preserving their close relationship with support varieties.  

For an affine group scheme $\bG$ equipped with an exponential map  $\cE_{(r)}: \bG_{a(r)} \times \cN_p(\fg) 
\to \bG_{(r)}$ of height $r$, there is a natural morphism \ $\Phi_{\fg,r}: \cC_r(\cN_p(\fg)) \ \to \ V_r(\bG)$
where $\cC_r(\cN_p(\fg))$ is the affine variety of $r$-tuples of $p$-nilpotent, pair-wise commuting 
elements of $\fg$ made explicit in Definition \ref{defn:cCr}.
We say that $\bG$ is of exponential type of height $r$ if  $\Phi_{\fg,r}$ is a homeomorphism
of scheme-theoretic points.  Examples include Frobenius kernels of various classical groups,
their parabolic subgroups, and unipotent subgroups of these parabolic subgroups.  (See, for 
example, \cite{Sob2}.)
For $\bG$ of exponential type of height $r$, Theorem \ref{thm:thetas} introduces $\Theta_{\fg,r}
\in k[\cC_r(\cN_p(\fg))] \otimes k\bG_{(r)}$ which very closely models $\Theta_{\bG,r}$.

This leads to our simpler $p$-nilpotent operator $\Theta_{\fg,r}^{exp}
\in k[\cC_r(\cN_p(\fg))]\ \otimes k\bG_{(r)}$ for $\bG$ of exponential type of height $r$, a 
linearization of $\Theta_{g,r}$.  A fundamental computation of \cite[Thm 1.12]{FPS} 
enables us to conclude that  $\Theta_{\fg,r}$ and $\Theta_{\fg,r}^{exp}$ determine the same
support varieties for finite dimensional $\bG_{(r)}$-modules.  By adopting the formalism
we have discussed for $JT_{\bG,r,M}(-)$, we establish the Jordan type functions
$JT_{\fg,r,M}(-)$ and $JT_{\fg,r,M}^{exp}(-)$ which satisfy similar properties.

\vskip .1in

We begin by defining the class of affine group schemes $\bG$ for which 
$r$-tuples of $p$-nilpotent, pair-wise commuting elements of the Lie algebra of $\bG$
determine height $r$ 1-parameter subgroups of $\bG$.

\begin{defn}
\label{defn:exp}
Let $\bG$ be an affine group scheme of finite type over $k$ with Lie algebra $\fg$.  
Fix a positive integer $r$.   We say that a morphism 
$$\cE_{(r)}:  \bG_{a(r)} \times \cN_p(\fg)  \quad \to \quad \bG_{(r)}$$
equips $\bG$ with the structure of an exponentiation of height $r$ if $\cE_{(r)}$
satisfies the following conditions for any commutative $k$-algebra $A$,  $p$-nilpotent 
elements $B , B^\prime \ \in A\otimes \fg$, 
and  elements $\alpha, \beta \in \bG_{a(r)}(A)$:
\begin{enumerate}
\item
The specialization of $\cE_{(r)}$ along $B: k[\cN_p(\fg)] \to A$, \ 
$\cE_{(r),B}:\bG_{a(r),A} \to \bG_{(r),A}$, \ is a 1-parameter subgroup of $\bG_{(r),A}$.
\item
If $B^\prime$ commutes with $B$, then
$$\cE_{(r),B}(\alpha)\cdot \cE_{(r),B^\prime}(\beta)= \cE_{(r),B^\prime}(\beta) 
\cdot \cE_{(r),B}(\alpha) \in \bG_{(r)}(A).$$
\item  
$\cE_{(r),\alpha \cdot B}(-) \ = \ \cE_{(r),B}( \alpha \cdot -): \bG_{a(r),A} \to \bG_{(r),A}$.
\end{enumerate}
\end{defn}  

\vskip .1in

As exhibited in Construction \ref{construct:cU}, the structure $\cE_{(r)}$ of exponentiation of height $r$
for $\bG$ enables us to ``exponentiate" $r$-tuples of $p$-nilpotent, pair-wise commuting elements
of $\fg$.  We first define the variety of such $r$-tuples, $\cC_r(\fg)$.

\begin{defn}
\label{defn:cCr}
Let $\bG$ be an affine group scheme of finite type over $k$ with Lie algebra $\fg$.
Consider the functor on commutative $k$ algebras sending $A$ to an $r$-tuple 
$\ul B = (B_0,\ldots,B_{r-1})$ of elements in $A\otimes \fg$ which satisfies the conditions that $B_s^{[p]} = 0$
for all $s, \ 0 \leq s < r$ and that $[B_s,B_{s^\prime}] = 0$ for all $s, s^\prime$ with  $0 \leq s,s^\prime < r$.
This functor is represented by an affine scheme $\cC_r(\fg)$.  
 In contrast to \cite[1.11.2]{F23}, we give $\cC_r(\fg)$ the monoid action $\bA^1 \times \cC_r(\fg) \to \cC_r(\fg)$
by sending $(\alpha;B_0,\ldots,B_{r-1})$ to $(\alpha\cdot B_0,\alpha^p\cdot B_1,\ldots,\alpha^{p^{r-1}}\cdot B_{r-1})$.  
\end{defn}

\vskip .1in

We now assume the $\bG$ is equipped with an exponential $\cE_{(r)}$ of height $r$.
The following construction of $\cE_{\fg,r}: \bG_{a(r)} \times \cC_r(\cN_p(\fg)) \ \to \ \bG_{(r)}$ 
given in (\ref{eqn:Er-def}) extends the exponential structure 
 $\cE_{(r)}: \bG_{a(r)} \times \cN_p(\fg)  \ \to \ \bG_{(r)}$.
 This leads to the ``Lie algebra analogue" 
$$\cU_{\fg,r}: \bG_{a(r)} \times \cC_r(\cN_p(\fg)) \quad  \to \quad \bG_{(r)} \times \cC_r(\cN_p(\fg))$$ 
of the universal 1-parameter subgroup 
$\cU_{\bG,r}: \bG_{a(r)} \times V_r(\bG) \to \bG_{(r)} \times V_r(\bG)$ of Definition \ref{defn:univ}.

\begin{construction}
\label{construct:cU}
Let $\bG$ be an affine group scheme of finite type over $k$ equipped with an exponentiation
$\cE_{(r)}:  \bG_{a(r)} \times \cN_p(\fg)  \to  \bG_{(r)}$
of height $r$.  For any $s, \ 0\leq s < r$, we define
\begin{equation}
\label{eqn:Ers-def} 
\cE_{r|s}  \quad \equiv \quad \cE_{(r)}  \circ (1\times pr_s) \circ (F^s \times 1):
 \bG_{a(r)}\times \cC_r(\fg) \ \to \ \bG_{(r)}
\end{equation}
where $pr_s$ is projection onto the $s$-th factor and $F^s$ is the $s$-th iterate of the Frobenius
morphism.  We denote by $\cE_{r|s,B}$ the restriction of $\cE_{r|s}$ along a (geometric) point
$B: k(B) \to \cN_p(\fg)$.

We define
\begin{equation}
\label{eqn:Er-def} 
\cE_{\fg,r} \quad \equiv \quad \prod \circ \ (\times_{s=0}^{r-1} \ \cE_{r|s}) \circ Diag
\end{equation} 
to be the composition 
$$\bG_{a(r)} \times \cC_r(\cN_p(\fg)) \ \to (\bG_{a(r)} \times \cC_r(\cN_p(\fg)))^{\times r} 
\ \to \ (\bG_{(r)})^{\times r} \ \to \bG_{(r)}
$$
where the product is taken with respect to the group structure of $\bG_{(r)}$.  

We define 
\begin{equation}
\label{eqn:cU}
\cU_{\fg,r|s} \equiv \cE_{r|s} \times pr_{\cC_r(\cN_p(\fg))}, \ \cU_{\fg,r} \equiv \cE_{\fg,r} \times pr_{\cC_r(\cN_p(\fg))}: 
\bG_{a(r)} \times \cC_r(\cN_p(\fg)) \ \to \ \bG_{(r)} \times \cC_r(\cN_p(\fg)),
\end{equation}
given by maps on coordinate agebras as the 
the $k[\cC_r(\fg)]$-linear extensions of $\cE_{r|s}^*, \ \cE_{\fg,r}^*$ along $k \to k[\cC_r(\cN_p(\fg))]$.
\end{construction}

\vskip .1in

The following proposition identifies the specializations of the maps $\cU_{\fg,r|s}$ and $\cU_{\fg,r}$ 
 at a geometric point of $\cC_r(\cN_p(\fg))$.  In particular,
 the specialization of $\cU_{\fg,r}$ at a geometric point $\ul B = (B_0,\ldots,B_{r-1})$ of $\cC_r(\cN_p(\fg))$ 
is expressed as a product of Frobenius twists of exponentiations of the $B_i$'s.  

\begin{prop}
\label{prop:specialize}
Retain the hypotheses and notation of  Definition \ref{defn:exp}. 
For any geometric point $\ul B = (B_0,\ldots,B_{r-1})$ of $\cC_r(\cN_p(\fg))$, the specialization of
$\cU_{\fg,r|s}$ at $\ul B$ is denoted by $\cU_{\fg,r|s,B_s}$.  Then
the specialization of $\cU_{\fg,r}$ at $\ul B$ is given by
\begin{equation}
\label{eqn:specialize-Urs}
\cU_{\fg,r,\ul B } \ = \ \prod_{s=0}^{r-1} \cU_{\fg,r|s,B_s} \circ F^s: \bG_{a(r),k(\ul B)}\ \to 
\bG_{a(r),k(\ul B)}\ \to \ \bG_{(r),k(\ul B)}.
\end{equation}
In particular, 
\begin{equation}
\label{eqn:special}
(\cU_{\fg,r,\ul B})_*: k(\ul B) \otimes k\bG_{a(r)} \to k(\ul B) \otimes k\bG_{(r)}, \quad
1\otimes u_{r-1} \ \mapsto \ \prod_{s=0}^{r-1} (\cE_{r|s,B_s})_*(u_{r-s-1}).
\end{equation}
\end{prop}

\begin{proof}
Equations (\ref{eqn:specialize-Urs}) and (\ref{eqn:specialize-Urs}) follow from equations (\ref{eqn:Ers-def}) and
(\ref{eqn:Er-def}) together with the definitions of $\cU_{\fg,r|s,\ul B}$ and $\cU_{\fg,r,\ul B }$ as 
$k[\cC_r(\cN_p(\fg))]$-linear extensions of $\cE_{r|s}, \ \cE_{\fg,r}$.

Equation (\ref{eqn:special}) follows immediately from (\ref{eqn:specialize-Urs}).
\end{proof}

\vskip .1in

We proceed to investigate the close relationship between $\cU_{\bG,r}$ and $\cU_{\fg,r}$.

\begin{prop}
\label{prop:action}
Let $\bG$ and $\cE_{(r)}$ satisfy the conditions of  Definition \ref{defn:exp}.  
\begin{enumerate}
\item
The morphisms 
$$\cU_{\fg,r|s}, \ \cU_{\fg,r}:  \bG_{a(r),k[\cC_r(\cN_p(\fg))] } \quad \to \quad \bG_{(r),k[\cC_r(\cN_p(\fg))]}$$
are 1-parameter subgroups.
\item
These 1-parameter  subgroups determine morphisms 
$$\Phi_{\fg,r|s}, \ \Phi_{\fg,r}:  \cC_r(\cN_p(\fg)) \quad \to \quad V_r(\bG)$$
with the property that their pull-backs of the universal 1-parameter group 
\ $\cU_{\bG,r}:  \bG_{a(r),k[V_r(\bG)] } \to \ \bG_{(r),k[V_r(\bG)]}$\
equal the 1-parameter subgroups \ $\cU_{\fg,r|s}, \ \cU_{\fg,r}$.
\item
The morphisms  $\Phi_{\fg,r|s}, \Phi_{\fg,r}$ commute with the monoid actions of $\bA^1$ on 
$\cC_r(\cN_p(\fg))$ and  $V_r(\bG)$ in the sense that each of the  following squares commute:
\begin{equation}
\label{eqn:restrict-diagram}
 \xymatrixcolsep{5pc}\xymatrix{
 \cC_r(\cN_p(\fg)) \times \bA^1 \ar[d]^{\Phi_{\fg,r|s} \times 1} \ar[r] &  \cC_r(\cN_p(\fg)) \ar[d]^{\Phi_{r,s}} \\
 V_r(\bG) \times \bA^1  \ar[r] &  V_r(\bG).
}
\end{equation}
\end{enumerate}
\end{prop}

\begin{proof}
As can be seen using Proposition \ref{prop:specialize} and 
the condition of Definition \ref{defn:exp}(2),
$\cU_{\fg,r|s,\ul B}$ \  and\ $\cU_{\fg,r,\ul B}$ are morphisms of group schemes
over $\ul B$. 
This implies assertion (1).  Assertion 2 follows from the universal property of 
$\cU_{\bG,r}:  \bG_{a(r),k[V_r(\bG)] } \quad \to \quad \bG_{(r),k[V_r(\bG)]}$.

To check the commutativity of (\ref{eqn:restrict-diagram}), it suffices to observe for any $\alpha
\in A$ and for any $A$-point  $\cC_r(\bG)$ of the form $(0,\ldots,B_s,\ldots 0)$ (and thus
of weight $p^s$) that $\cE_{(r),\alpha\cdot \ul B} (t) \ = \cE_{(r),(0,\ldots,\alpha^{p^s}B_s,\ldots 0)} (t)
= \cE_{(r),\ul B}(\alpha^{p^s}t) = (\alpha\cdot \cE_{(r),\ul B})(t)$.
\end{proof}

\vskip .1in

In \cite{SFB1}, Suslin, Bendel, and the author introduced the condition that an 
embedding $\bG \hookrightarrow GL_N$ 
of linear algebraic groups be of exponential type; they observe that various algebraic groups satisfy
this condition, including many classical simple groups, their parabolic subgroups, and unipotent radicals 
of these parabolic subgroups. 
Subsequent generalizations of this condition on $\bG$ occur in \cite{F15} and \cite{F23}.  
A more definitive account of such linear algebraic groups can be found in
P. Sobaje's article \cite{Sob2}.  The following definition is a natural generalization of those
earlier considerations of groups of exponential type.  The motivating special case is the exponential
function of $GL_N$ applied to a $p$-nilpotent element of $\fg l_N$.

\begin{defn}
\label{defn:exp-type}
Let $\bG$ be an affine group scheme of finite type over $k$ equipped with an
exponentiation $\cE_{(r)}: \bG_{a(r)} \times \cN_p(\fg)  \ \to \ \bG_{(r)}$ of height $\leq r$.
 We say that $\bG$ equipped with $\cE_{(r)}$ is  {\it of exponential type of height $r$}
if the morphism $ \Phi_{\fg,r}: \cC_r(\fg) \ \to \ V_r(\bG)$ of Theorem \ref{prop:action}(1)
induces a bijection on geometric points.   
\end{defn}

\vskip .1in

In \cite[Defn 1.6]{F91}, the notions of a ``continuous algebraic map" and a ``bicontinuous morphism" 
were introduced in order to study rational maps between schemes which induce well defined maps
on geometric points.   In that work, the topologies of interest were either the \`etale topology or 
the analytic topology (in the case that the ground field was the complex numbers).  We are now
working with varieties over $k$, a field of characteristic $p > 0$.

\begin{lemma}
\label{lem:bicont}
 Let $(\bG, \cE_{(r)})$ be an affine group scheme of \it exponential type of height $r$.  
 The bicontinuous morphism
  \begin{equation}
  \label{eqn:Phi}
 \Phi_{\fg,r}: \cC_r(\cN_p(\fg)) \quad \to \quad V_r(\bG)
 \end{equation}
 is a homeomorphism of Zariski topological spaces which commutes with $\bA^1$-actions.
\end{lemma}

\begin{proof}
The assumption of exponential type of height $r$ implies that $\Phi_{\fg,r}$ is a bijection on scheme-theoretic
points.  Since $\Phi_{\fg,r}$ is a morphism, it is continuous.  Since $\bP \Phi_{\fg,r}: \bP \cC_r(\cN_p(\fg))
 \to \bP V_r(\bG)$ is 
proper, it is a closed map, thus induces a bijection on closed subsets, thus a homeomorphism.  This 
implies that the pull-back/restriction $\cC_r(\cN_p(\fg))  - \{ 0 \} \to V_r(\bG) - \{ 0 \}$ is a homeomorphism.
Since the cone point $0$ is closed in both $\cC_r(\cN_p(\fg))$ and $V_r(\bG)$, this implies 
that $\Phi_{\fg,r}$ is also a homeomorphism.
\end{proof}

\vskip .1in

We adapt the construction of $\Theta_{\bG,r}$ in Definition \ref{defn:univ-Theta} and its use
in Theorem \ref{thm:theta} by replacing $\cU_{\bG,r}$ by $\cU_{\fg,r}$.

\begin{thm}
\label{thm:thetas}
 Let $(\bG, \cE_{(r)})$ be an affine group scheme of \it exponential type of height $r$. 
 We define the $p$-nilpotent operators
\begin{equation}
\label{eqn:theta-exp}
\Theta_{\fg,r|s} \ \equiv \ (\cU_{\fg,r|s})_*(1\otimes u_{r-1}), \ \Theta_{\fg,r} \ \equiv \ (\cU_{\fg,r})_*(1\otimes u_{r-1}) \
 \in \ k[\cC_r(\cN_p(\fg)))] \otimes k\bG_{(r)},
\end{equation}
and denote by $\Theta_{\fg,r|s,\ul B}, \  \Theta_{\fg,r,\ul B}$ their specializations at a scheme-theoretic point
$\ul B \in \cC_r(\cN_p(\fg))$.
\begin{enumerate}
\item
$\Theta_{\fg,r}  \ = \ (\Phi^*_{\fg,r} \otimes 1)(\Theta_{\bG,r})$.
\item
$\Theta_{\fg,r|s,\ul B} \  = \ (\cE_{r|s,B_s})_*(u_{r-s-1}), \ \Theta_{\fg,r,\ul B} \ = \prod_{s=0}^{r-1} \Theta_{\fg,r|s,\ul B} 
\quad \in \ k(\ul B) \otimes k\bG_{(r)},$
elements of the group algebra of $\bG_{(r),k(\ul B)}$.
\item
Both $\Theta_{\fg,r|s}$ and $\Theta_{\fg,r}$ are homogeneous of degree $p^{r-1}$ with 
respect to the grading on $k[\cC_r(\cN_p(\fg))]$.
\item
 For any $\bG_{(r)}$-module $M$, tensor product with $\Theta_{\fg,r|s}$ and with $\Theta_{\fg,r}$
determines $p$-nilpotent, $k[\cC_r(\cN_p(\fg))]$-linear operators
$$\Theta_{\fg,r|s,M}, \ \Theta_{\fg,r,M}: k[\cC_r(\cN_p(\fg))] \otimes M \ \to \ k[\cC_r(\cN_p(\fg))] \otimes M.$$
\item
The base changes along a scheme theoretic point $\ul B \in \cC_r(\cN_p(\fg))$ of the operators 
$\Theta_{\fg,r|s,M}, \ \Theta_{\fg,r,M}$  
are given by the actions of $\Theta_{\fg,r|s,\ul B}, \  \Theta_{\fg,r,\ul B} $  on $k(\ul B) \otimes M$.  
\end{enumerate}
\end{thm}

\begin{proof}
Observe that the elements $\Theta_{\fg,r|s}, \  \Theta_{\fg,r}$ are $p$-nilpotent because $1\otimes u_{r-1}
\in \bG_{a(r),\cC_r(\cN_p(\fg))]}$ is $p$-nilpotent.
Assertion (1) follows from Proposition \ref{prop:action}(2), the definition of $\Theta_{\bG,r}$ in 
Definition \ref{defn:univ-Theta}, and the definition of $\Theta_{\fg,r}$ given above.
The first part of assertion (2) identifying $\Theta_{\fg,r|s,\ul B}$ follows from Proposition \ref{prop:specialize},
so that the second part is verified using the fact that specialization along $k[\cC_r(\cN_p(\fg))] \to k(\ul B)$,
$k[\cC_r(\cN_p(\fg))]\otimes \bG_{(r)} \to k(\ul B) \otimes \bG_{(r)}$,
commutes with the product map of $\bG_{(r)}$.

Assertion (3) follows from Theorem \ref{thm:theta}(5).
Assertions (4) and (5) and  their proofs are direct analogues of Theorem \ref{thm:theta}(3),(4).
\end{proof}

\vskip .1in

As an immediate corollary of Theorem \ref{thm:thetas}(1), we have the following.

\begin{cor}
\label{cor:aff-supp}
 Let $(\bG, \cE_{(r)})$ be an affine group scheme of \it exponential type of height $r$.
Consider  a geometric point $\ul B$ of $\cC_r(\fg)$ with image $x = \Phi_{\fg,r}(\ul B)$.
 For any finite dimensional  $\bG_{(r)}$-module $M$, 
 $$JT(\Theta_{\fg,r,M})(\ul B)) \quad = \quad JT(\Theta_{\bG,r,M})(x).$$
 In other words,
 $$JT(\Theta_{\fg,r,M})(-) \ = \ JT(\Theta_{\bG,rM}(-) \circ \Phi_{\fg,r}(-): 
 \bP \cC_r(\fg) \ \to \ \cY,$$
 for any finite dimensional $\bG_{(r)}$-module $M$.
\end{cor}

\vskip .1in

We denote by $\bP \cC_r(\cN_p(\fg))$ the weighted projective variety $\Proj k[\cC_r(\cN_p(\fg))]$, parallel
to the notation of $\bP V_r(\bG)$ for $\Proj k[V_r(\bG)]$.   We omit the proof of the following
elementary lemma.

\begin{lemma}
A geometric point of $\Spec K \to \bP \cC_r(\cN_p(\fg))$ corresponds to an equivalence 
class $\langle \ul B \rangle$ of
non-zero $r$-tuples $\ul B = (B_0,\ldots,B_{r-1})$ of $p$-nilpotent, pair-wise commuting elements 
of $K \otimes \fg$, where $(B_0,\ldots,B_{r-1})$ is equivalent to $(B_0^\prime,\ldots,B_{r-1}^\prime)$ if and only 
it there exists some $0 \not= \alpha \in K$ such that $B_i^\prime = \alpha^{p^i} B_i$ for each $i \ 0 \leq i < r$.

Under the bicontinuous homeomorphism $\Phi_{\fg,r}$ of Lemma \ref{lem:bicont}, the projection \\
$V_r(\bG) {\text -} \{0 \} \to \bP V_r(\fg)$ corresponds to 
$\cC_r(\cN_p(\fg)) {\text -} \{0 \} \to \bP\cC_r(\cN_p(\fg))$ sending $\ul B$ to $\langle \ul B \rangle$.
\end{lemma}

\vskip .1in

We next formulate the projectivization $\tilde\Theta_{\fg,r}$ of $\Theta_{\fg,r}$ and give 
the analogue of Theorem \ref{thm:global-univ} for $\tilde \Theta_{\fg,r}$.
As in Theorem \ref{thm:global-univ}, 
we retain the notation and 
conventions introduced prior to Proposition \ref{prop:weighted} relating graded modules
for a graded commutative ring $R$ and sheaves on $\Proj R$.

\vskip .1in

\begin{prop}
\label{prop:proj-fg}
Let $(\bG, \cE_{(r)})$  be an affine group scheme of exponential  type of height $r$.
The  operator  $\Theta_{\fg,r}$ of Theorem \ref{thm:thetas} determines a global section 
$$\tilde  \Theta_{\fg,r} \quad \in \quad  
\Gamma(\bP \cC_r(\cN_p(\fg)),\cO_{\bP\cC_r(\cN_p(\fg))}(p^{r-1})\otimes k\bG_{(r)})$$
of the coherent, locally free sheaf $\cO_{\bP\cC_r(\cN_p(\fg))}(p^{r-1})\otimes k\bG_{(r)}$ 
of $\cO_{\bP\cC_r(\cN_p(\fg))}$-modules which satisfies the following properties.
\begin{enumerate}
\item
For any $\bG_{(r)}$-module $M$, tensor product with $\tilde  \Theta_{\fg,r}$ determines the 
map of quasi-coherent $\cO_{\bP\cC_r(\cN_p(\fg))}$-modules
$$\tilde \Theta_{\fg,r,M}: (k[\cC_r(\cN_p(\fg))]\otimes M)^{\sim} \quad \to \quad 
(k[\cC_r(\cN_p(\fg))]\otimes M)^\sim(p^{r-1}).$$
\item
For any non-zero divisior $h \in  k[\cC_r(\cN_p(\fg))]_{p^{r-1}}$, let $U_h$ denote the affine open subset \
$Spec (k[\cC_r(\cN_p(\fg))][h^{-1}])_0 \ \subset \bP\cC_r(\cN_p(\fg))$. 
The natural  isomorphism of $(\cO_{\bP\cC_r(\cN_p(\fg))})_{U_h}$-modules
$$h^{-1} \otimes 1: (\cO_{\bP\cC_r(\cN_p(\fg))}(p^{r-1}) \otimes k\bG_{(r)})_{|U_h} \quad 
\stackrel{\sim}{\to} \quad  (\cO_{\bP\cC_r(\cN_p(\fg))} \otimes k\bG_{(r)})_{|U_h}$$
sends 
$$(\tilde \Theta_{\fg,r})_{|U_h} \ \in \ \Gamma(U_h,\cO_{\bP\cC_r(\cN_p(\fg))}(p^{r-1}) \otimes k\bG_{(r)}) \ = \
(k[\cC_r(\cN_p(\fg))][1/h])_{p^{r-1}} \otimes k\bG_{(r)} $$
to 
$$(h^{-1}\otimes 1)\cdot\Theta_{\fg,r} \ \in \ k[\cC_r(\cN_p(\fg))][1/h]\otimes k\bG_{(r)}.$$
\item
If $\ul B$ is a scheme theoretic point of $\cC_r(\cN_p(\fg))$ projecting to $\langle \ul B \rangle$ 
such that $h(\langle \ul B \rangle) \not= 0$, then the restriction 
along $\Spec k(\ul B) \to \bP\cC_r(\cN_p(\fg))$ of $\tilde  \Theta_{\fg,r}$ is given by 
\begin{equation}
\label{eqn:invert}
k(\ul B) \otimes_{\cO_{\bP\cC_r(\cN_p(\fg))}} \tilde\Theta_{\fg,r} \quad = 
\quad c\cdot \prod_{s=0}^{r-1} (\cE_{r|s,B_s})_*(1\otimes u_{r-s-1}) \ \in \ k(\ul B) \otimes k\bG_{(r)}
\end{equation}
for some $0 \not= c \in k(\ul B)$.
\end{enumerate}
\end{prop}

\begin{proof}
By Proposition \ref{prop:weighted}, $\cO_{\bP\cC_r(\cN_p(\fg))}(p^{r-1})\otimes k\bG_{(r)}$ 
is a coherent, locally free
sheaf of $\cO_{\bP\cC_r(\cN_p(\fg))}$-modules on $\bP \cC_r$.
Proposition \ref{prop:action}(3) implies that $(\Phi_{r|s})^*: k[V_r(\bG)] \to k[\cC_r(\cN_p(\fg))]$ is a 
map of graded $k$-algebras.   Thus, Theorem 
\ref{thm:theta}(5) and Theorem \ref{thm:thetas}(1) tell us that $\Theta_{\fg,r} \in k[\cC_r(\cN_p(\fg))] \otimes k\bG_{(r)}$ 
is homogenous of degree $p^{r-1}$. Therefore, $\Theta_{\fg,r}$ determines 
$\tilde \Theta_{\fg,r}$ in 
 $\Gamma(\cO_{\bP\cC_r(\cN_p(\fg))},\cO_{\bP\cC_r(\cN_p(\fg))}(p^{r-1})\otimes k\bG_{(r)})$.

The proofs of assertions (1), (2), and (3) are verified by evident analogues for 
$\tilde \Theta_{\fg,r,}$ of the proofs given in 
the proof of Theorem \ref{thm:global-univ}
for the corresponding statements for $\tilde \Theta_{\bG,r,}$
\end{proof}

\vskip .1in

We proceed to simplify $\Theta_{\fg,r}$ by replacing the product in (\ref{eqn:invert}) by a sum.

\begin{defn}
\label{defn:simplify}
Let $(\bG, \cE_{(r)})$ be an affine group scheme of \it exponential type of height $r$.  We define
$$\Theta^{exp}_{\fg,r} \ \equiv \ \sum_{s=0}^{r-1} \Theta_{\fg,r|s} \quad \in \quad  
k[\cC_r(\cN_p(\fg))]\otimes k\bG_{(r)}.$$
Thus, for each $\ul B \in \cC_r(\cN_p(\fg))$
 $$ \Theta^{exp}_{\fg,r,\ul B} \ 
 \equiv \ \sum_{s=0}^{r-1} \Theta_{\fg,r|s,\ul B} \ = \ \sum_{s=0}^{r-1} (\cE_{r|s,B_s})_*(1\otimes u_{r-s-1})
 \quad \in \quad k(\ul B) \otimes k\bG_{(r)}.$$
\end{defn}

\vskip .1in

The justification of replacing $\Theta_{\fg,r}$ by  $\Theta_{\fg,r}^{exp}$ is implicit in the following
proposition, a reformulation of \cite[Prop 4.3]{F15} (following a result of Sobaje \cite[Prop 2.3]{Sob}).

\begin{prop}
\label{prop:linear}
 Let $(\bG, \cE_{(r)})$ be an affine group scheme of \it exponential type of height $r$,
 let $M$ be a finite dimensional $\bG_{(r)}$-module, and let $\ul a = \sum_{i=1}^p a_i [i]$
 be a Jordan type which is maximal among the set of Jordan types $\{ JT(\Theta_{\fg,r,\ul B}, k(\ul B)\otimes M); \ 
 \ul B \in \cC_r(\bG) \}$ for $M$.
 For any scheme-theoretic point $\ul B \in \cC_r(\cN_p(\fg))$, 
$$JT(\Theta^{exp}_{\fg,r,\ul B}, k(\ul B)\otimes M) \ = \ \ul a \quad
 {\text if \ and \ only \ if} \quad  JT(\Theta_{\fg,r,\ul B}, k(\ul B)\otimes M)
 \ = \ \ul a.$$
\end{prop}

\begin{proof}
In \cite[Prop 2.3]{Sob}, Sobaje regards $\Theta_{\fg,r,\ul B} \ = \prod_{s=0}^{r-1} \Theta_{\fg,r|s,\ul B}$
as the composition
$$(\prod_{s=0}^{r-1} \circ (\otimes_{s=0}^{r-1} (\cE_{r|s,B_s})_*) \circ (1\otimes \Delta): 
\ k(\ul B) \otimes k\bG_{a(r)} \ \to \ k(\ul B) \otimes k\bG_{a(r)}^{\otimes r}  $$
$$ \quad \to \ k(\ul B) \otimes k\bG_{(r)}^{\otimes r} \to k(\ul B) \otimes k\bG_{(r)}$$
applied to $1\otimes u_{r-1} \in k(\ul B) \otimes k\bG_{a(r)}$.  Using this composition, 
he verifies that  $\Theta_{\fg,r,\ul B}$ equals $\sum_{s=0}^{r-1} \Theta_{\fg,r|s,\ul B}$ plus
other terms which are scalar multiples of products of at least two factors of this sum,
factors which are $p$-nilpotent and pair-wise commuting.  The proposition then follows from
\cite[Thm 1.12]{FPS}  (see also \cite[Prop 8]{CLN} and  \cite[Lem 6.4]{SFB2}).
\end{proof}

\vskip .1in

As shown in Theorem \ref{thm:gen-strata} and remarked in Remark \ref{rem:generalized}, the
support variety and its generalizations
for a finite dimensional $\bG_{(r)}$-module $M$ are determined by $JT_{\bG,r,M}(-)$
and thus by $JT_{\fg,r,M}(-)$.  (See Corollary \ref{cor:aff-supp}.)  

Since these support varieties are formulated in terms of maximality of Jordan types, 
Proposition \ref{prop:linear} has the following corollary.

\vskip .1in

\begin {cor}
\label{cor:gen-supp}
 Let $(\bG, \cE_{(r)})$ be an affine group scheme of \it exponential type of height $r$ and
 let $M$ be a finite dimensional $\bG_{(r)}$-module.  Then the support variety and 
 generalized support varieties of $M$  interpreted as closed subvarieties of $\cC_r(\fg)$ using
 $JT_{\bG,r,M}(-)$ are left unchanged by replacing  $JT_{\bG,r,M}(-)$ by  $JT^{exp}_{\bG,r,M}(-)$.
 \end{cor}
 
\vskip .1in

\begin{remark}
\label{remark:leq2}
By definition,  $\Theta^{exp}_{\fg,1} \ = \ \Theta_{\fg,1}$.  As observed in \cite[Prop 2.3]{Sob} (see also the proof
of \cite[Prop 6.5]{SFB2}), $\Theta^{exp}_{\fg,2} \ = \ \Theta_{\fg,2}$.  On the other hand, for $r > 2$,
$\Theta^{exp}_{\fg,r} \ \not= \ \Theta_{\fg,r}$.
\end{remark}

\vskip .1in

We next give the evident definition of the ``exponential Jordan type function" of a $\bG_{(r)}$-module, parallel to the 
definition of $JT_{\bG,r,M}(-)$ given in Definition \ref{defn:JT}.

\begin{defn}
\label{defn:local-exp}
 Let $(\bG, \cE_{(r)})$ be an affine group scheme of exponential type of height $r$ and let $M$ 
 be a finite dimensional $\bG_{(r)}$-module.  We define
 $$JT^{exp}_{\fg,r,M}(-):  \cC_r(\cN_p(\fg)) \quad \to \quad \cY$$
 by sending a geometric point $\ul B: \Spec k(\ul B) \to \cC_r(\cN_p(\fg))$ to \  
 \begin{equation}
 \label{eqn:formula-exp}
 JT(\Theta^{exp}_{\fg,r,M})(\ul B) \quad = 
 \quad JT(\sum_{s=0}^{r-1} (\cE_{r|s,B_s})_*(1\otimes u_{r-s-1}),k(\ul B)\otimes M))\ ,
 \end{equation}
 the Jordan type of the operator $\Theta^{exp}_{\fg,r,\ul B}$ on $k(\ul B)\otimes M$.
 
 Observe that the operator $(\cE_{r|s,B_s})_*(1\otimes u_{r-s-1})$ on $k(\ul B)\otimes M$ can be identified
 with the operator $1\otimes u_{r-s-1} \in k(\ul B)\otimes k\bG_{a(r)}$ on $(\cE_{r|s,B_s})^*(k(\ul B)\otimes M)$.
 \end{defn}

\vskip .1in

We next formulate the projectivization $\tilde\Theta_{\fg,r}^{exp}$ of $\Theta_{\fg,r}^{exp}$ and give 
the analogues of Theorem \ref{thm:global-univ} for $\tilde \Theta_{\bG,r}$ and 
Proposition \ref{prop:proj-fg} for $\tilde \Theta_{\fg,r}$.

\vskip .1in

\begin{prop}
\label{prop:proj-exp}
Let $(\bG, \cE_{(r)})$  be an affine group scheme of exponential  type of height $r$.
The  operator  $\Theta_{\fg,r}^{exp}$ of Definition \ref{defn:simplify} determines a global section 
$$\tilde  \Theta_{\fg,r}^{exp} \quad \in \quad  \Gamma(\cO_{\bP\cC_r(\cN_p(\fg))},
\cO_{\bP\cC_r(\cN_p(\fg))}(p^{r-1})\otimes k\bG_{(r)})$$
of the coherent, locally free sheaf $\cO_{\bP\cC_r(\cN_p(\fg))}(p^{r-1})\otimes k\bG_{(r)}$ 
of $\cO_{\bP\cC_r(\cN_p(\fg))}$-modules which
satisfies the following properties.
\begin{enumerate}
\item
For any $\bG_{(r)}$-module $M$, tensor product with $\tilde  \Theta_{\fg,r}^{exp}$ determines the 
map of quasi-coherent $\cO_{\bP\cC_r(\cN_p(\fg))}$-modules
$$\tilde \Theta_{\fg,r,M}^{exp}: (k[\cC_r(\cN_p(\fg))]\otimes M)^{\sim} \quad \to 
\quad (k[\cC_r(\cN_p(\fg))]\otimes M)^\sim(p^{r-1}).$$
\item
For any non-zero divisior $h \in  k[\cC_r(\cN_p(\fg))]_{p^{r-1}}$, let $U_h$ denote the affine open subset \
$Spec (k[\cC_r(\cN_p(\fg))][h^{-1}])_0 \ \subset \bP \cC_r(\fg)$. 
The natural  isomorphism of $(\cO_{\bP\cC_r(\cN_p(\fg))})_{U_h}$-modules
$$h^{-1} \otimes 1: (\cO_{\bP\cC_r(\cN_p(\fg))}(p^{r-1}) \otimes k\bG_{(r)})_{|U_h} \quad 
\stackrel{\sim}{\to} \quad  (\cO_{\bP\cC_r(\cN_p(\fg))} \otimes k\bG_{(r)})_{|U_h}$$
sends 
$$(\tilde \Theta_{\fg,r}^{exp})_{|U_h} \ \in \ 
\Gamma(U_h,\cO_{\bP\cC_r(\cN_p(\fg))}(p^{r-1}) \otimes k\bG_{(r)}) \ = \
(k[\cC_r(\cN_p(\fg))][1/h])_{p^{r-1}} \otimes k\bG_{(r)} $$
to 
$$(h^{-1}\otimes 1)\cdot\Theta_{\fg,r}^{exp} \ \in \ k[\cC_r(\cN_p(\fg))][1/h]\otimes k\bG_{(r)}.$$
\item
If $\ul B$ is a scheme theoretic point of $\cC_r(\cN_p(\fg))$ projecting to $\langle \ul B \rangle$ 
such that $h(\langle \ul B \rangle) \not= 0$, then the restriction 
along $\Spec k(\ul B) \to \bP\cC_r(\cN_p(\fg))$ of $\tilde  \Theta_{\fg,r}^{exp}$ is given by 
\begin{equation}
\label{eqn:invert-exp}
k(\ul B) \otimes_{\cO_{\bP\cC_r(\cN_p(\fg))}} \tilde\Theta_{\fg,r}^{exp} \quad = 
\quad c\cdot \sum_{s=0}^{r-1} (\cE_{(r),B_s})_*(1\otimes u_{r-s-1}) \ \in \ k(\ul B) \otimes k\bG_{(r)}
\end{equation}
for some $0 \not= c \in k(\ul B)$.
\item
For any $\bG_{(r)}$-module $M$, the restriction of $\tilde \Theta_{\fg,r,M}^{exp}$ along 
$\Spec k(\ul B) \to \bP\cC_r(\cN_p(\fg))$
as in (3) is a non-zero scalar multiple of the image of \\
$\sum_{s=0}^{r-1} (\cE_{r|s,B_s})_*(1\otimes u_{r-s-1})$ in $End_{k(\ul B)}(k(\ul B) \otimes M)$.
\end{enumerate}
\end{prop}

\begin{proof}
By Proposition \ref{prop:weighted}, $\cO_{\bP\cC_r(\cN_p(\fg))}(p^{r-1})\otimes k\bG_{(r)}$ 
is a coherent, locally free
sheaf of $\cO_{\bP \cC_r(\cN_p(\fg))}$-modules on $\bP\cC_r(\cN_p(\fg))$.
Proposition \ref{prop:action}(3) implies that $(\Phi_{r|s})^*: k[V_r(\bG)] \to k[\cC_r(\cN_p(\fg))]$ is a 
map of graded $k$-algebras so that $\Theta_{\fg,r}^{exp}$ is homogenous by Theorem 
\ref{thm:theta}(5) and Theorem \ref{thm:thetas}(1).  Thus, $\Theta_{\fg,r}^{exp}$ determines 
$\tilde \Theta_{\fg,r,M}^{exp}$ in \\
 $\Gamma(\cO_{\bP\cC_r(\cN_p(\fg))},\cO_{\bP\cC_r(\cN_p(\fg))}(p^{r-1})\otimes k\bG_{(r)})$.

The proofs of assertions (1) - (4) are verified by evident analogues for $\tilde \Theta^{exp}_{\fg,r,}$ 
of the proofs given in the proof of Theorem \ref{thm:global-univ}
for the corresponding statements for $\tilde \Theta_{\bG,r,}$
\end{proof}

\vskip .1in

The following theorem makes explicit the definition of the Jordan type function $JT_{\fg,r,M}^{exp}(-)$
and states properties analogous to properties for J$T_{\bG,r,M}^(-)$ given in Theorem \ref{thm:JT}.
 The proof of Theorem \ref{thm:JT-exp}
follows from Proposition \ref{prop:proj-exp} exactly as the proof of Theorem \ref{thm:JT} follows
from Theorem \ref{thm:global-univ}.

\begin{thm}
\label{thm:JT-exp}
Let $(\bG, \cE_{(r)})$  be an affine group scheme of exponential  type of height $r$.
For any finite dimensional $\bG_{(r)}$-module $M$  and any 
 scheme-theoretic point  $\ul B$ of $\cC_r(\cN_p(\fg))$ projecting to $\langle \ul B \rangle \in \bP\cC_r(\cN_p(\fg))$,
we define 
$$JT_{\fg,r,M}^{exp}(\langle \ul B \rangle) \quad \equiv \quad 
JT(\sum_{s=0}^{p-1}(\cE_{r|s,B_s})_*(1\otimes u_{r-s-1}),k(\ul B)\otimes M),$$
the Jordan type of the $p$-nilpotent element \ 
$k(\ul B) \otimes_{\cO_{\bP\cC_r(\cN_p(\fg))}} \tilde\Theta_{\fg,r}^{exp} \in k(\ul B) \otimes \bG_{(r)}$
acting on $k(\ul B)\otimes M$.
\begin{enumerate}
\item
For any finite dimensional $\bG_{(r)}$-module $M$,
the association $ \langle \ul B \rangle \ \mapsto  JT_{\fg,r,M}^{exp}(\langle \ul B \rangle)$ 
determines a continuous map
$$JT_{\fg,r,M}^{exp}(-): \bP\cC_r(\cN_p(\fg)) \quad \to \quad \cY.$$ 
\item
For  $\bG_{(r)}$-modules $M, N$ of dimensions $m,n$ and any  scheme-theoretic point 
$\ul B  \in \bP\cC_r(\cN_p(\fg))$, 
$$JT_{\fg,r,M\oplus N}^{exp}( \langle \ul B \rangle) \ = \ JT_{\fg,r,M}^{exp}( \langle \ul B \rangle) + JT_{\fg,r,N}^{exp}( \langle \ul B \rangle)  \  \in \ \cY_{m+n}.$$
\item
Assume either  that $r=1$ or that $ \langle \ul B \rangle \in \bP\cC_r(\cN_p(\fg))$ is a generic point.  
For $\bG_{(r)}$-modules $M, N $  of dimensions $m, n$, 
$$JT_{\fg,r,M\otimes N}^{exp}( \langle \ul B \rangle) \ = 
\ JT_{\fg,r,M}^{exp}( \langle \ul B \rangle) \otimes JT_{\fg,r,N}^{exp}( \langle \ul B \rangle)  \  \in \ \cY_{m\cdot n}.$$
\item
If $f: M \to N$ is a surjective map of $\bG_{(r)}$-modules, then $JT_{\fg,r,N}^{exp}( \langle \ul B \rangle) \ 
\leq JT_{\fg,r,M}^{exp}( \langle \ul B \rangle)$
for any $ \langle \ul B \rangle \in \bP\cC_r(\cN_p(\fg))$;  in other words, $JT_{\fg,r,N}^{exp}(-) \leq JT_{\fg,r,M}^{exp}(-)$.
\item
A short exact sequence $ 0 \to N \ \to \ M \ \to \ Q \to 0$ of finite dimensional $\bG_{(r)}$-modules
 is locally split if and only if $JT_{\fg,r,M}^{exp}(-) = JT_{\fg,r,N}^{exp}(-) +JT_{\fg,r,Q}^{exp}(-).$
 \item
If  $\phi: \bH \to \bG$ is a morphism of affine group schemes of exponential type of height $r$, then 
for any finite dimensional $\bG_{(r)}$-module $M$
$$JT_{\fh,r,\phi*(M)}^{exp}(-): \bP \cC_r(\cN_p(\fh)) \ \to \ \cY$$ equals
$$JT_{\fg,r,M}^{exp}(-) \circ \phi_*: \bP \cC_r(\cN_p(\fh \ \to \ \bP\cC_r(\cN_p(\fg)) \ \to \ \cY.$$
\end{enumerate}
\end{thm}

\vskip .2in


\section{Examples of exponential Jordan types}
\label{sec:apps}

We provide some examples of computations of  the simplified
Jordan type function $JT^{exp}_{\fg,r,M}(-)$ of Theorem \ref{thm:JT-exp}. 
One challenge to computing Jordan types is the behavior of Jordan types with
respect to the sum of two commuting, $p$-nilpotent operators.  Another challenge
 is the fact that for $r > 1$ the Jordan type functions 
$JT_{\bG,r,M}(-), \ JT_{\fg,r,M}^{exp}(-)$ rarely commute with tensor products of 
$\bG_{(r)}$-modules at all points of $ \bP \cC_r(\fg)$.

\vskip .1in
\begin{prop}
By Proposition \ref{prop:linear}, Examples \ref{ex:constant-type} and \ref{ex:Carlson-mod} 
remain valid if $(\bG, \cE_{(r)})$  is an affine group scheme of exponential  type of height $r$
 and if $JT_{\bG,r,M}(-)$ is replaced by $JT^{exp}_{\fg,r,M}(-)$.
The justification of Example \ref{ex:perp} also justifies its  validity
when $JT_{\bG,r,M}(-)$ is replaced by $JT^{exp}_{\fg,r,M}(-)$.

Example \ref{ex:ht1} remains unchanged, since $JT_{\bG,1,M}(-) \ = \ JT^{exp}_{\fg,1,M}(-)$.
\end{prop}

\vskip .1 in

On the other hand, as we see in the following examples, some computations simplify if we replace 
$JT_{\bG_{(r)},M}(-)$ by $JT^{exp}_{\fg,M}(-)$.

\vskip .1in

\begin{ex}
\label{ex:Ga}
Adopt the notation of Examples \ref{ex:weight} and \ref{ex:Theta}.
Let $V$ be an $n$-dimensional $k$-vector space, and let $\alpha_0,\ldots \alpha_{r-1}$ be 
be pair-wise commuting $p$-nilpotent endomorphisms of $V$.  Denote by $M$ the 
$(\bG_a^{\times r})_1$-module such that $x_i \in k(\bG_a^{r})_1$ acts as $\alpha_i$ and denote
by $N$ the $\bG_{a(r)}$-module such that $u_i \in k\bG_{a(r)}$ also acts as $\alpha_i$.

Recall from Example \ref{ex:weight} the morphism
$\Phi: \bP^{r-1} \to w\bP(1,p,\ldots,p^{r-1} )$.  
Since $\Theta_{\bG_{a,r}}^{exp} \ = \ 
\sum_{i=0}^{r-1} T_i^{p^{r-1-i}} \otimes u_i \ \in \ k[V_r(\bG_{a(r)})] \otimes k\bG_{a(r)}$, 
$$\Phi^*(\Theta_{\bG_a, r}^{exp}) \quad = \quad \sum_{i=0}^{r-1} t_i^{p^{r-1}} \otimes x_i 
\ = \ F^{r-1*}(\Theta_{\bG_a^{\times r},1}).$$ 
Thus,
$$JT_{\fg_a,r,N}^{exp}(-)\circ \Phi \ = \ JT_{\bG_a^{\times r},1,M} \circ F^{r-1}: \
\bP^{r-1} \ \to \ \cY.$$
\end{ex}

\vskip .1in

In the following examples, we use (\ref{eqn:formula-exp}) of Definition \ref{defn:local-exp}
which becomes 
$$JT^{exp}_{\fg,r,M}(\ul B) \quad = \quad JT(\sum_{s=0}^{r-1} (exp_{B_s})_*(1\otimes u_{r-s-1}),k(\ul B))$$
provided that $(\bG, \cE_r)$  is associated with an embedding $i: \bG \hookrightarrow GL_N$ 
of {\it exponential type of height $r$}.

\vskip .1in

\begin{ex}
\label{ex:SL2,2}
Set $\bG = SL_{2(2)}$, so that $\fg = sl_2$.  We consider the irreducible $\bG$-module 
$$M \ = \ S(\lambda_0) \otimes S(\lambda_1)^{(1)}, \quad 0 \leq \lambda_0, \lambda_1
< p, \  \lambda = \lambda_0 + p\lambda_1.$$
  We proceed to compute 
$JT^{exp}_{sl_2,2,M}(-):  \cC_2(\cN_p(sl_2)) \to \cY_{\leq}$ as in Definition \ref{defn:local-exp}.
To compute $JT_{SL_2,2,M}(\ul B)$
for $B_0,B_1$ commuting elements in $sl_2$, we appeal to Proposition \ref{prop:adjoint} 
to reduce to the case that $\ul B = (a_0 E, a_1 E)$ where $E$ is the $2\times 2$ matrix with
$E_{1,2} = 1$ and $E_{1,1} = E_{2,1} = E_{2,2} = 0$.  We may then view $\cE_{(2),B_s}$
as factoring as $\bG_{a(2),k(\ul B)} \to \bG_{a(2),k(\ul B)} \subset SL_{2(2),k(\ul B)}$.

To compute  $JT_{sl_2,2,S(\lambda)}^{exp}(\ul B)$, we identify the Jordan type of 
$$(\cE_{(2),B_0})_*(u_1) + (\cE_{(2),B_1})_*(u_0) =   a_0 E \cdot u_1 + a_1^p E \cdot u_0 
\in k(\ul B) \otimes kSL_{2(2)}$$
on $S(\lambda_0) \otimes S(\lambda_1)^{(1)}$.  
Here, the action of $a_0 E \cdot u_1$ on $S(\lambda_0)$
is zero, whereas the action of $a_1^p E \cdot u_0$ on $S(\lambda_0)$ is $a_1^p$ times the action of
$E \in sl_2$ on $S(\lambda_0)$;  the action of $a_0 E \cdot u_1$ on $S(\lambda_1)^{(1)}$
is $a_0$ times the action of $E \in sl_2$ on $S(\lambda_1)$ and the action of $a_1^p E \cdot u_0 $
on $S(\lambda_1)^{(1)}$ is zero.

Observe that the Jordan type of $E$ on $S(\lambda_0) = [m]$ where $m= \lambda_0 +1$
and on $S(\lambda_1) = [n]$ where $n = \lambda_1+1$.  
Thus, if $a_0 \not= 0, \ a_1 = 0$, then the Jordan type $S(\lambda)$ is $m\cdot [n]$; and if 
$a_0 = 0, \ a_1 \not= 0$, then the Jordan type $S(\lambda)$ is $n\cdot [m]$.
More generally,  $S(\lambda_0) \otimes S(\lambda_1)^{(1)}$ as a $k(\ul B) \otimes k[u]/u^p$-module
with $1\otimes u$ acting as $a_0 E \cdot u_1 + a_1^p E \cdot u_0$ can be identified 
with a tensor product of two $k[u]/u^p$-modules: the first with 
with $u$ acting as $a_1^p E \cdot u_0$ on $S(\lambda_0)$, the second with $u$
acting as  $ a_0 E \cdot u_1$ on $S(\lambda_1)^{(1)}$.  Consequently, if both
$a_0, \ a_1$ are non-zero, then the the Jordan type
of $a_1^p E \cdot u_0 + a_0 E \cdot u_1$ on $S(\lambda)$ equals $[m]\otimes [n]$.
\end{ex}

\vskip .1in

\begin{ex}
\label{ex:GLN}
Let $V$ be a polynomial representation of $GL_N$ of degree $< p$
and consider the $i$-th Frobenius twist $V^{(i)}$ for some $i \geq 0$.  
We fix some $r > i$ and proceed to determine $JT_{\fg l_N,r,V^{(i)}}(\ul B)$
for $\ul B = (B_0,\ldots,B_{r-1})$, an $r$-tuple of $p$-nilpotent, pairwise commuting 
elements of $k(\ul B) \otimes \fg l_N$.
Then the action of $(\exp_{B_s})_*(u_{r-s-1})$ on $V^{(i)}$ is non-zero
if and only if $i = r-s-1$, in which case this action equals the action of 
$B_s \in k(\ul B) \otimes \fg l_N$ 
 acting on $k(\ul B) \otimes V$.  Consequently, $JT_{\fg l_N,r,V^{(i)}}^{exp}(\ul B) = 
JT(B_s,V)$ where $s = r-i-1$.

More generally, we can take a polynomial representation $V_s$ of $GL_N$ for 
each $s, \ 0 \leq s< r$ with each $V_s$
of degree $< p$ and consider $M = \otimes_{s= 1}^{r-1} V_s^{(s)}$.  
Let $\ul B$ be a geometric point of $\cC_r(\fg l_N)$.  
The action of $u = \sum_{s=0}^{r-1}(exp_{B_s})_*(1\otimes u_{r-s-1})$
on $k(\ul B)\otimes M$ is the tensor product of the actions of 
$B_s \in k(\ul B) \otimes \fg l_N$  acting on $k(\ul B) \otimes V_s$.

Thus,  $JT_{\fg l_N,r,M}^{exp}(\ul B) \ = \ \bigotimes_{s=0}^{r-1} JT(B_{r-s-1},V_s).$
\end{ex}

\vskip .1in

\begin{ex}
\label{ex:Theta-GL_N}
We employ Theorem 3.7 of \cite{FP3} to consider the general situation of $(\bG, \cE_{(r)})$ 
an affine group scheme of exponential type of height $r$ and $M$ a $\bG$ module of
dimension $m$.  Namely, if  $\rho: \bG \to GL_m$ determines $M$, then 
$(\rho\circ \cE_{(r),B_s})_*(1\otimes u_{r-s-1})$ acting on
$k(\ul B) \otimes M$ is given by the $m\times m$ matrix whose $(i,j)$-entry
equals the coefficient of  $t^{p^{r-s-1}} \in k(\ul B)\otimes \bG_{a(r)}$ of $(\cE_{r|s,B_s})_*(1\otimes u_{r-s-1})$ 
applied to $\rho^*(X_{i,j}).$

Thus $JT^{exp}_{\fg,r,M}(\ul B)$ is the Jordan type of the sum (indexed by $s, 0 \leq s < r$) 
of these $m\times m$ matrices.  
\end{ex}

\vskip .2in


\section{Vector bundles and stratifications}
\label{sec:bundles}

In this section, we revisit the construction of vector bundles associated to 
 $\bG_{(r)}$-modules of constant Jordan type (and more general constant $j$-type).   
 We refer to \cite{CFP} for an introduction to modules of constant
Jordan type and for numerous examples of these interesting $\bG$-modules.
In \cite{FP3}, such vector bundles were associated
to $\bG_{(r)}$-modules of constant Jordan type in order to provide invariants
which can distinguish some  $\bG_{(r)}$-modules whose Jordan type functions are equal.
 
 As we point out, the constructions of \cite{FP3} when applied to an arbitrary $\bG_{(r)}$-module
$M$ yield coherent sheaves on $\bP V_r(\bG)$ whose restrictions to each stratum associated
 to  $JT_{\bG,r,M}(-)$ are locally free.

\vskip .1in

\begin{defn}
Let $\bG$ be an affine group scheme of finite type over $k$, let $r$ be a chosen positive integer,
$j$ and integer with  $1\leq j < r$,
and let $M$ be a $\bG_{(r)}$-module.  We define the following coherent sheaves on $\bP V_r(\bG)$
$$\cKer\{ \tilde\Theta_{\bG,r,M}^j\}, \quad \cCoker\{ \tilde\Theta_{\bG,r,M}^j\}, \quad \cIm\{ \tilde\Theta_{\bG,r,M}^j\}$$
as the sheaves associated to kernel, cokernel, and image of the $j$-th iterate of the map of
graded $k[V_r(\bG)]$-modules
$$\tilde \Theta_{\bG,r,M}: (k[V_r(\bG)]\otimes M)^{\sim} \quad \to \quad (k[V_r(\bG)]\otimes M)^\sim(p^{r-1})$$
given in Theorem \ref{thm:global-univ} (1).
\end{defn}

\vskip .1in

\begin{defn}
\label{defn:sheaves}
Let $(\bG, \cE_{(r)})$ be an affine group scheme of exponential type of height $r$, 
$j$ and integer with  $1\leq j < r$, and let $M$ be a $\bG_{(r)}$-module.  We define the following 
coherent sheaves on $\bP \cC_r(\fg)$
$$\cKer\{ \tilde\Theta_{\fg,r,M}^j\}, \quad \cCoker\{ \tilde\Theta_{\fg,r,M}^j\}, \quad \cIm\{ \tilde\Theta_{\fg,r,M}^j\}$$
$$\cKer\{ (\tilde\Theta_{\fg,r,M}^{exp})^j\}, \quad \cCoker\{(\tilde\Theta_{\bG,r,M}^{exp})^j\}, \quad
 \cIm\{ (\tilde\Theta_{\bG,r,M}^{exp})^j\}$$
as the sheaves associated to kernel, cokernel, and image of the $j$-th iterate of the maps 
$\tilde\Theta_{\fg,r,M}$ and $\tilde\Theta_{\fg,r,M}^{exp}$ of 
graded $k[\cC_r(\fg)]$-modules given in Proposition \ref{prop:proj-fg} (1) and Proposition \ref{prop:proj-exp} (1).
\end{defn}

\vskip .1in

We next verify that the coherent sheaves of Definition \ref{defn:sheaves} are locally constant
on each stratum.

\begin{thm}
\label{thm:restr}
Let $\bG$ be an affine group scheme of finite type of $k$ and let $M$ be a finite dimensional
$\bG_{(r)}$-module.   Then for each Jordan type $\ul a \ = \ \sum_{i=1}^p a_i [i]$, the restrictions of 
the coherent sheaves $\cKer\{ \tilde\Theta_{\bG,r,M}^j\}, \ \cCoker\{ \tilde\Theta_{\bG,r,M}^j\}, 
\ \cIm\{ \tilde\Theta_{\bG,r,M}^j\}$ to the locally closed subvariety 
$JT_{\bG,r,M}^{-1}(\ul a) \ \subset \ \bP V_r(\bG)$ are locally free coherent sheaves on $JT_{\bG,r,M}^{-1}(\ul a)$.

If $(\bG, \cE_{(r)})$ is an affine group scheme of exponential type of height $r$, then 
the pull-backs along $\Phi_{\fg,r}: JT_{\fg,r,M}^{-1}(\ul a) \to JT_{\bG,r,M}^{-1}(\ul a)$
of  the sheaves
$ker\{ \tilde\Theta_{\bG,r,M}^j\}, \\ coker\{\tilde\Theta_{\bG,r,M}^j\}, \
 im\{ \tilde\Theta_{\bG,r,M}^j\}$
 are locally free sheaves on $JT_{\fg,r,M}^{-1}(\ul a) \ \subset \ \bP \cC_r(\fg)$.

Moreover, if $(\bG, \cE_{(r)})$ is an affine group scheme of exponential type of height $r$, then 
the restrictions of the coherent sheaves 
$\cKer\{ (\tilde\Theta_{\fg,r,M}^{exp})^j\}, \ \cCoker\{(\tilde\Theta_{\fg,r,M}^{exp})^j\}, \\
 \cIm\{ (\tilde\Theta_{\fg,r,M}^{exp})^j\}$
 to $(JT_{\fg,r,M}^{exp})^{-1}(\ul a) \ \subset \ \bP \cC_r(\fg)$
 are locally free sheaves for any Jordan type $\ul a \ = \ \sum_{i=1}^p a_i [i]$.
 \end{thm}
 
 \begin{proof}
 Let $Y_{\ul a}$ denote $(JT_{\fg,r,M}^{exp})^{-1}(\ul a) \ \subset \ \bP V_r(\fg)$ and consider its closure
 $$\ol {Y_{ \ul a}} \ \equiv \ JT_{\fg,r,M}^{-1}(\ol{\{ \ul a \}}) \ = \ \bigcup_{\ul b \leq \ul a} JT_{\fg,r,M}^{-1}(\ul b)
 \quad \hookrightarrow \quad \bP V_r(\bG)$$
 corresponding to the homogeneous ideal $I_{\ul a} \ \subset \ k[V_r(\bG)]$.  Consider
 $k[\ol {Y_{ \ul a}}] \ = \ k[V_r(\bG)]/I_{\ul a}$, equipped with the grading inherited from the grading of $k[V_r(\bG)]$
 and set 
 $$\Theta_{\bG,r;\ol{Y_{\ul a}}} \ \equiv \ k[\ol {Y_{ \ul a}}] \otimes_{k[V_r(\bG)]} \Theta_{\bG,r}.$$
 So defined, $\Theta_{\bG,r;\ol{Y_{\ul a}}}$ determines a $p$-nilpotent $k[\ol {Y_{ \ul a}}]$-linear operator 
 $\Theta_{\bG,r,M:\ol{Y_{\ul a}}}$ on
 $k[\ol {Y_{ \ul a}}] \otimes M$ as in Theorem \ref{thm:theta}.
 As in Theorem \ref{thm:global-univ}, $\Theta_{\bG,r;\ol{Y_{\ul a}}}$ determines
 $$\tilde \Theta_{\bG,r;\ol{Y_{\ul a}}} \ \in \ \Gamma(\ol{Y_{\ul a}},\cO_{\ol{Y_{\ul a}}}(p^{r-1}) \otimes k\bG_{(r)})$$
 and the map of graded $k[\ol {Y_{ \ul a}}]$-modules 
 $$\tilde\Theta_{\bG,r,M;\ol{Y_{\ul a}}}:  k[\ol {Y_{ \ul a}}] \otimes M)^\sim \ \to \ (k[\ol {Y_{ \ul a}}]
  \otimes M)^\sim(p^{r-1}).$$
 Then the restrictions of $\cKer\{ \tilde\Theta_{\bG,r,M}^j\}, \ \cCoker\{ \tilde\Theta_{\bG,r,M}^j\}, 
\ \cIm\{ \tilde\Theta_{\bG,r,M}^j\}$ along $Y_{\ul a} \subset \ol{Y_{ \ul a}} \ \hookrightarrow \ \bP V_r(\bG)$
 are identified with the restrictions to the open subset $Y_{\ul a} \subset \ol{Y_{ \ul a }}$ of 
$\cKer\{ \tilde\Theta_{\bG,r,M;\ol{Y_{\ul a}}}^j\}, \ \cCoker\{ \tilde\Theta_{\bG,r,M;\ol{Y_{\ul a}}}^j\}, 
\ \cIm\{ \tilde\Theta_{\bG,r,M;\ol{Y_{\ul a}}}^j\}$.

The proof that the restrictions to $Y_{\ul a}$ of $\cKer\{ \tilde\Theta_{\bG,r,M;\ol{Y_{\ul a}}}^j\}, 
\ \cCoker\{ \tilde\Theta_{\bG,r,M;\ol{Y_{\ul a}}}^j\}$ and 
$\cIm\{ \tilde\Theta_{\bG,r,M;\ol{Y_{\ul a}}}^j\}$
are locally free coherent sheaves on $Y_{\ul a}$ is completed by repeating the proof
of \cite[Thm 4.13]{FP3} with $U \subset \bP (G), \ \tilde \Theta_G$ of that theorem replaced by 
$Y_{\ul a} \ \subset \ \ol{Y_{ \ul a}}, \ \tilde \Theta_{\bG,r;\ol{Y_{\ul a}}}$
as above.

If $(\bG, \cE_{(r)})$ is an affine group scheme of exponential type of height $r$,
then the pull-backs of these locally free sheaves along along the homeomoprhism
$\Phi_{\fg,r}: JT_{\fg,r,M}^{-1}(\ul a) \to JT_{\bG,r,M}^{-1}(\ul a)$ are also locally free. 
Finally, only notational changes replacing $\Theta_{\fg,r,M}$ by $\Theta_{\fg,r,M}^{exp}$
are needed to obtains proofs of the assertions for $\cKer\{ (\tilde\Theta_{\fg,r,M}^{exp})^j\}, \ 
\cCoker\{(\tilde\Theta_{\fg,r,M}^{exp})^j\}, 
 \cIm\{ (\tilde\Theta_{\fg,r,M}^{exp})^j\}$.
 \end{proof}
 
 \vskip .1in
 
 Example \ref{ex:constant-type} gives examples of $\bG_{(r)}$ modules leading to a stratification
 of $\bP V_r(\bG)$ with a single stratum and  Example \ref{ex:Carlson-mod} gives examples of 
 $\bG_{(r)}$-modules with exactly two strata.  
 
 We give one ``trivial" example of a $\bG_{(r)}$ module with more than two strata, primarily to show 
 that some examples can be easily computed.
 
 \begin{ex}
 \label{ex:stratify}
 Let $\bG = \bG_a^{\times p}$, so that  $k\bG_{(1)} = k[u_0,\dots,u_{p-1}]$.  Consider 
 the $k\bG_{(1)}$-module $M$ defined as the quotient of $k\bG_{(1)}$ by the ideal
 generated by $\{ u_i^i \}$.  Consider points $\ul b = \langle b_0,\ldots, b_{p-1} \rangle \in 
 \bP V_1(\bG) \simeq \bP^{p-1}$.
  If $\ul b$ is given by $b_i = 1, \ b_j = 0, j \not= i$, then 
  $JT_{\bG,1,M}(\ul b)$ equals $\sum_{j\not= i} [j] + i[1]$.
 \end{ex}
 
 \vskip .1in
 
 We next investigate the relationship between the coherent sheaves 
 $\cKer\{ \tilde\Theta_{\fg,r,M}^j\}, \\ \cCoker\{\tilde\Theta_{\bG,r,M}^j\}, \
 \cIm\{ \tilde\Theta_{\bG,r,M}^j\}$
and the coherent sheaves $\cKer\{ (\tilde\Theta_{\fg,r,M}^{exp})^j\},\\
\cCoker\{(\tilde\Theta_{\bG,r,M}^{exp})^j\}, \ \cIm\{ (\tilde\Theta_{\bG,r,M}^{exp})^j\}$.

\vskip .1in

\begin{defn}
\label{defn:Theta/P1}
Let $(\bG, \cE_{(r)})$ be an affine group scheme of exponential type of height $r$.
Consider the bigraded algebra $k[s,t] \otimes k[\cC_r(\cN_p(\fg))]$ 
where both $s\otimes 1$ and $t\otimes 1$ are given 
the bigrading  $(1,0)$.  We define $\Theta_{\fg/\bP^1,r}$ to be
\begin{equation}
\label{eqn:theta-homotopy}
\Theta_{\fg/\bP^1,r} \quad \equiv \quad (s\otimes 1)(1\otimes  \Theta_{\fg,r}^{\exp}) \ + 
(t\otimes 1) (1\otimes \Phi_{\fg,r}^*(\Theta_{\bG_{(r)} })) \ \in \  k[s,t] \otimes k[\cC_r(\cN_p(\fg))] \otimes k\bG_{(r)}.
\end{equation}
So defined, $\Theta_{\bG/\bP^1,r}$ is bihomogenous of bidegree $(1,p^{r-1})$, thereby determining a 
global section of $\cO_{\bP^1 \times \bP\cC_r(\cN_p(\fg))}(1,p^{r-1})$.
\end{defn}
\vskip .1in

\begin{prop}
\label{prop:overP1}
Retain the notation and hypotheses of Definition \ref{defn:Theta/P1}.  Let $M$ be a finite dimensional
$\bG_{(r)}$-module.  For any scheme theoretic point $(\ul B, \langle s,t \rangle) \in \cC_r(\fg) \times \bP^1$,
$$JT(\Theta^{exp}_{\fg,r,\ul B}, k(\ul B)\otimes M) \ = \ \frac{m}{p}[p]  \ \ \iff \\ 
JT(\Theta_{\fg/\bP^1,r,(\ul B,\langle s,t \rangle)}, k(\ul B,\langle s,t \rangle)\otimes M)
 \ = \ \frac{m}{p}[p].$$
\end{prop}

\begin{proof}
Observe that 
$$\Theta_{\fg/\bP^1,r,(\ul B,\langle s,t \rangle)} \ = \ s \cdot \Theta_{\fg,r,\ul B}^{\exp} \ + \
t \cdot \Theta_{\fg,r,\ul B} \quad \in \quad k(\ul B,\langle s,t \rangle) \otimes k\bG_{(r)}.$$
As recalled in the proof of Proposition \ref{prop:linear}, $\Theta_{\fg,r,\ul B} - \Theta_{\fg,r,\ul B}^{\exp}$ 
is a linear combination of terms each of which is a product of at least 2 $p$-nilpotent factors.  Thus,
for any $\langle s,t \rangle \in \bP^1$, $s \cdot \Theta_{\fg,r,\ul B}^{\exp} \ + \
t \cdot \Theta_{\fg,r,\ul B}$ differs from $\Theta_{\fg,r,\ul B}^{\exp}$ by a sum of terms
each of which is a product of at least 2 $p$-nilpotent factors.   Consequently, the
proposition follows from \cite[Prop 8]{CLN}.
\end{proof}

\vskip .1in

If $(\bG, \cE_{(r)})$  be an affine group scheme of exponential type of height $r$, if $j$ is a positive
integer between 1 and $r-1$,  and if $M$ is a finite dimensional $\bG_{(r)}$-module, then $M$
 is said to have constant $j$-rank if for every scheme-theoretic point $\ul B \in \cC_r(\cN_p(\fg))$ the
 rank of $\Theta_{\fg,r,\ul B}^j: k(\ul B) \otimes M \to k(\ul B) \otimes M$ is 
 independent of $\ul B$.  (See \cite[Defn 3.12]{FP3}.)
 Thus, $M$ has constant Jordan type if and only if $M$ has constant $j$-type for each $j, \ 1 \leq j < r$.

\begin{prop}
\label{prop:bundle-equiv}
Let $(\bG, \cE_{(r)})$  be an affine group scheme of exponential type of height $r$,
and  let $M$ be a
$\bG_{(r)}$-module.  Then for any $j, 0 < j < p$,
there are quasi-coherent sheaves ${\cKer}\{ (\tilde \Theta_{\fg/\bP^1,r,M})^j \}, \
{\cCoker}\{ (\tilde \Theta_{\fg/\bP^1,r,M})^j \}, \ {\cIm}\{ (\tilde \Theta_{\fg/\bP^1,r,M})^j \}$
on $\bP^1 \times \bP\cC_r(\cN_p(\fg))$ satisfying
\begin{enumerate}
\item
The fibers over $\langle 0, t \rangle \in \bP^1$ of these sheaves are naturally identified with 
the sheaves on $\bP\cC_r(\cN_p(\fg))$ given by the pull-backs along
$\Phi_{\fg,r}^*$ of the corresponding sheaves on $\bP V_r(\bG)$ for $\tilde \Theta_{\bG,r,M}^j$.
\item 
The fibers over $\langle s, 0 \rangle \in \bP^1$ of these sheaves are naturally identified with 
the corresponding sheaves on $\bP\cC_r(\cN_p(\fg))$ for $(\tilde \Theta_{\bG,r,M}^{exp})^j$.
\item
If $M$ is a finite dimensional $\bG_{(r)}$-module of constant $j$-rank, then 
these sheaves are  locally free, coherent sheaves
on $\cC_r(\cN_p(\fg)) \times \bP^1$. 
\end{enumerate}

In other words, for any $\bG_{(r)}$-module of
constant $j$-type, taking kernels,  cokernels, and images of 
$(\tilde \Theta_{\bG,r,M})^j$ and $(\tilde \Theta_{\fg,r,M}^{exp})^j$
yield $\bA^1$-homotopy equivalent vector bundles on $\bP\cC_r(\cN_p(\fg))$.  
\end{prop}

\begin{proof}
By Proposition \ref{prop:action} and  Theorem \ref{thm:thetas},  $\Theta_{\fg/\bP^1,r}$ of 
Definition \ref{defn:Theta/P1} is bihomogeneous
of bidegree $(1,p^{r-1})$.  Thus, the bigraded module
\begin{equation}
\label{eqn:ker-A1} 
ker\{ (\Theta_{\fg/\bP^1,r})^j: (k[s,t] \otimes k[\cC_r(\cN_p(\fg))] \otimes M)
\ \to \ (k[s,t] \otimes k[\cC_r(\cN_p(\fg))] \otimes M)[1,p^{r-1}]
\end{equation}
 determines the quasi-coherent sheaves ${\cKer}\{ (\tilde \Theta_{\fg/\bP^1,r,M})^j \}, \
{\cCoker}\{ (\tilde \Theta_{\fg/\bP^1,r,M})^j \}, \\  {\cIm}\{ (\tilde \Theta_{\fg/\bP^1,r,M})^j \}$
on $\bP^1 \times \bP \cC_r(\fg)$.  The first two assertions follow from the observation that the fiber above 
$\langle a,b \rangle \in \bP^1$ of a sheaf on $\bP^1 \times \bP\cC_r(\cN_p(\fg))$ is the sheaf associated
to the bigraded graded module (\ref{eqn:ker-A1}) obtained by extending along $k[s,t] \to 
k(\langle a, b \rangle) \in \bP^1$ which sends $s$ to $a$ and $t$ to $b$.

For $M$ a finite dimensional $\bG_{(r)}$-module of constant $j$-type, the verification of \cite[Thm 4.13]{FP3} 
showing that $\cKer\{ \tilde\Theta_{\bG,r,M}^j\}, \ \Coker\{ \tilde\Theta_{\bG,r,M}^j\}, 
\ \cIm\{ \tilde\Theta_{\bG,r,M}^j\}$ are locally free sheaves of $\bP V_r(\bG)$ applies 
equally to verify that
the sheaves $\cKer\{ (\tilde\Theta_{\fg,r,M}^{exp})^j\}, \\ \cCoker\{(\tilde\Theta_{\bG,r,M}^{exp})^j\}, \
\cIm\{ (\tilde\Theta_{\bG,r,M}^{exp})^j\}$ are locally free sheaves on $\bP\cC_r(\cN_p(\fg))$, and  applies 
 with only notational changes to show that
 ${\cKer}\{ (\tilde \Theta_{\fg/\bP^1,r,M})^j \}, \
{\cCoker}\{ (\tilde \Theta_{\fg/\bP^1,r,M})^j \}, \\  {\cIm}\{ (\tilde \Theta_{\fg/\bP^1,r,M})^j \}$
are locally free sheaves on $\bP^1\times \bP\cC_r(\cN_p(\fg))$
\end{proof}

\vskip .1in

\begin{remark}
As shown in \cite{BP}, most vector bundles on $\bP^{r-1}$ can be realized by kernel, cokernel, 
and image sheaves
associated to modules of constant Jordan type for elementary abelian $p$-groups of rank $r$.
For each such vector bundle on $\bP^{r-1}$, there is a corresponding vector bundle on $w\bP(1,p,\ldots,p^{r-1})$
associated to a module of constant Jordan type for  $\bG_{a(r)}$.
It would be interesting to find examples in the context of Proposition \ref{prop:bundle-equiv}(3)
of locally free sheaves on $\bP \cC_r(\fg)$ which are $\bA^1$-equivalent but not isomorphic.

One possible source of such examples might come from $\bG = \bG_{a(r)}$, $I \subset k\bZ/p^{\times r}$
the augmentation of $k\bZ/p^{\times r}$, and $M$ the $\bG$-module corresponding 
to the $k\bZ/p^{\times r}$-module $I^m/I^n$ with $n > m$.  In this case, we would be comparing
$\bA^1$-equivalent vector bundles on the singular variety $w\bP(1,p,\ldots,p^{r-1})$.
An approach might be to look at local invariants of these bundles at singular points of
$w\bP(1,p,\ldots,p^{r-1})$.
\end{remark}

\vskip .2in


\section{Jordan types for linear algebraic groups}
\label{sec:lin-alg-grp}

In  \cite[Defn 4.4]{F15}, we formulated the local Jordan type of  a finite dimensional rational $\bG$-module $M$ for 
$\bG$ a linear algebraic group equipped with a structure of exponential type.  
There is some subtlety to this construction in order to ``package"  the Jordan types of $M_{|\bG_{(r)}}$
of restrictions of $M$ to all Frobenius kernels of $\bG$.  This involves reordering $r$-tuples of $p$-nilpotent, 
pairwise commuting elements
$(B_0,\ldots, B_{r-1})$ of $\fg$ so that they appear as $(B_{r-1},\ldots, B_0)$.
The definition of the Jordan type of a finite dimensional rational $\bG$-module $M$ is then formulated
in terms of the action of a colimit of $p$-nilpotent operators, rather than as a colimit of Jordan 
types of restrictions $M_{|\bG_{(r)}}$ of $M$ to the Frobenius kernels $\bG_{(r)}$ of $\bG$.

In this final section, we observe that this approach leads to Jordan type functions
for finite dimensional rational $\bG$-modules $M$ provided that 
$\bG$ is of ``exponential type of infinite height".  In Definition \ref{defn:JT-Gmod}, we extend the construction of 
 $JT^{exp}_{\fg,r,M}(-):  \cC_r(N_p(\fg)) \ \to \ \cY$ of Definition \ref{defn:local-exp} for $\bG_{(r)}$-modules
  to $JT_{\bG,\infty,M}^{exp}(-):\cC_\infty(\fg) \to \cY$ for finite dimensional
 rational $\bG$-modules.   As shown in Proposition \ref{prop:stable} (correcting an error of 
 \cite[Prop 4.8]{F15}), $JT_{\bG,\infty,M}^{exp}(-)$ is a twisted stabilization of 
 $JT_{\bG,r,M_{|\bG_{(r)}}}^{exp}(-)$.

\vskip .1in

We say that an affine group scheme $\bG$ of finite type over $k$ is a linear algebraic group
if $\bG$ is reduced and irreducible.   In what follows, $k\bG \ \equiv \varinjlim_r k\bG_{(r)}$, the 
``distribution algebra" of $\bG$.  We define $\cC_\infty(\cN_p(\fg))$ to be the ind-scheme
$$\cC_\infty(\cN_p(\fg)) \ \equiv \ \varinjlim_r \cC_r(N_p(\fg)); \quad  \cC_{r}(N_p(\fg)) \to \cC_{r+1}(\cN_p(\fg)),$$
where $(B_0,\ldots, B_{r-1}) \mapsto (B_0,\ldots, B_{r-1},0).$
We define the ``coordinate algebra" 
$$ k[\cC_\infty(\cN_p(\fg))]  \quad \equiv \quad \varinjlim_r k[\cC_r(\cN_p(\fg))].$$
We define a scheme-theoretic (respectively, geometric) point of  $\cC_\infty(\cN_p(\fg))$  to be an
element in the direct limit of  scheme-theoretic (respectively, geometric) points of  $ \{ \cC_r(\cN_p(\fg)) \}$.
We can view such a point $\iota_x: \Spec K \to \cC_\infty(\cN_p(\fg))$ as given by a ``continuous" map of 
algebras $\iota_x^*: k[\cC_\infty(\cN_p(\fg))] \to K$.

\vskip .1in

We provide analogues of the constructions given in Construction \ref{construct:cU}.

\begin{defn}
\label{defn:infinite-ht}
Let $\bG$ be a linear algebraic group with Lie algebra $\fg$, equipped with a morphism
$$\cE: \bG_a \times \cN_p(\fg) \quad \to \quad \bG$$ 
such that  \ $(\bG_{(r)}, \ \cE_{|r})$ \ constitutes an affine group scheme of exponential 
type of height $r$ for each $r > 0$.
Here,  $\cE_{|r}: \bG_{a(r)} \times \cN_p(\fg) \to \bG_{(r)}$ is the restriction of $\cE$.
We say that $\bG$ is of exponential type of infinite height.

For any (geometric) point $B$ of $\cN_p(\fg)$, we denote by 
$\cE_{|r,B}: k(B) \times \bG_{a(r)} \ \to \ k(B) \otimes \bG$ the restriction of $\cE_{|r}$
along $\Spec k(B) \to \cN_p(\fg)$.

We define 
$$\cE_{\infty|r} \ = \ \cE_{|r} \circ (1\times pr_s): \bG_{a(r)} \times \cC_\infty(\cN_p(\fg)) \ \to \
\bG_{a(r)} \times \cN_p(\fg) \ \to \ \bG_{(r)}$$
and set 
$$\cU_{\fg,\infty|r} \ \equiv \ \cE_{\infty|r}\times pr_{\cC_\infty(\cN_p(\fg))}: 
\bG_{a(r)} \times \cC_\infty(\cN_p(\fg)) \ \to \ \bG_{(r)} \times \cC_\infty(\cN_p(\fg)).$$

We define the operator (a formal limit of well defined partial sums) 
\begin{equation}
\label{eqn:theta-infty}
\Theta_{\fg,\infty}^{exp} \ \equiv \ \sum_{s\geq 0} (\cU_{\fg,\infty|s})_*(1\otimes u_s) \ 
= \sum_{s\geq 0} \in \ 
k[\cC_\infty(\cN_p(\fg))] \otimes k\bG.
\end{equation}
\end{defn}

\vskip .1in

As verified in \cite[Prop 2.6]{F15}, for any linear algebraic group $(\bG,\cE)$ of exponential 
type of infinite height,  any finite dimensional $\bG$-module $M$,
any element $m \in M$, and any $B \in \cN_p(\fg)$, the element 
$(\cE_{|r,B_s})_*(1\otimes u_s)$ in $k[\cC_\infty(\cN_p(\fg))] \otimes k\bG$ acts trivially on 
$m$ for all sufficiently large $s$.  
This justifies the following proposition which should be compared with Definition \ref{defn:local-exp}.

\begin{prop}
\label{defn:Theta-formal} 
For any finite dimensional $\bG$-module $M$,  the formal sum $\Theta_{\fg,\infty}^{exp}$ of (\ref{eqn:theta-infty}) 
determines a well defined $p$-nilpotent $k[\cC_\infty(\fg)]$-linear operator 
$$\Theta_{\fg,\infty}^{exp}: k[\cC_\infty(\cN_p(\fg))] \otimes M \ \to \ k[\cC_\infty(\cN_p(\fg))] \otimes M.$$

The specialization at a scheme-theoretic point $\ul B \in \cC_\infty(\cN_p(\fg))$ of $\Theta_{\fg,\infty}^{exp}$ 
determines
the operator   
$$\Theta_{\fg,\infty,\ul B}^{exp} \ = \ \sum_{s \geq 0} (\cE_{|s,B_s})_*(1\otimes u_{s-1}): 
k[\ul B] \otimes M \ \to \ k[\ul B] \otimes M ,$$
(denoted $\sum_{s \geq 0} (\cE_{B_s})_*(u_s)$ in \cite{F15}).
\end{prop}

\vskip .1in

The following definition of $JT_{\bG,\infty,M}(\ul B)$ is essentially that of \cite[Defn 4.4]{F15}.

\begin{defn}
\label{defn:JT-Gmod} 
For any finite dimensional $\bG$-module $M$, 
we define  the Jordan type of $M$ at $\ul B \in \cC_\infty(\fg)$ to be 
\begin{equation}
\label{eqn:JT-infty}
JT_{\fg,\infty,M}^{exp}(\ul B) \quad \equiv \quad 
JT(\sum_{s \geq 0} (\cE_{|s,B_s})_*(1\otimes u_s),k(\ul B)\otimes M),
\end{equation}
the Jordan type of the $p$-nilpotent action of \ $ k(\ul B) 
\otimes_{k[\cC_\infty(\cN_p(\fg))]} \Theta_{\fg,\infty}^{exp} \in k(\ul B)\otimes k\bG$ \
acting on $k(\ul B)\otimes M$.

Thus, (\ref{eqn:JT-infty}) determines the Jordan type function
$$JT_{\fg,\infty,M}^{exp}(-):  \ \cC_\infty(\cN_p(\fg)) \quad \to \quad \cY.$$
\end{defn}

\vskip .1in 

We point out that \cite[Prop 4.8]{F15} is wrong as stated.  The following proposition provides a correction.
The proof of Proposition \ref{prop:stable} is an immediate consequence of the observation
above that $(\cE_{|r,B_s})_*(1\otimes u_s)$ in $k[\cC_\infty(\fg)] \otimes k\bG$ acts trivially on $m \in M$ 
 for all sufficiently large $s$.

\begin{prop}
\label{prop:stable}
Adopt the notation and hypotheses of Definition \ref{defn:JT-Gmod}.  If $M$ is finite dimensional, 
then 
$$JT_{\bG,\infty,M}^{exp}(-) \ = \ JT_{\bG,r,M_{|\bG_{(r)}}}^{exp}(-) \circ  \Lambda^r:  \  \cC_\infty(\cN_p(\fg))
\ \to \ \cC_r(\cN_p(\fg)) \ \to \ \cY, \quad r \gg 0,$$
where $\Lambda^r: \cC_\infty(\cN_p(\fg)) \to \cC_r(\cN_p(\fg))$ is given by 
sending $(B_0, \ldots, B_n,\ldots )$ to $(B_{r-1},\ldots,B_0)$.

In particular, $JT_{\bG,\infty,M}(-)$ is continuous.
\end{prop}

\vskip .1in

One can interpret the Proposition \ref{prop:stable} as saying for $M$ a finite dimensional 
$\bG$-module that the Jordan type function $JT_{\bG,\infty,M}^{exp}(-)$ is a twisted stabilization
 (with respect to increasing $r$) of the exponential Jordan type functions $JT_{\bG,r,M_{|\bG_{(r)}}}^{exp}(-)$ 
 of Definition {\ref{defn:local-exp}.

\vskip .2in


\end{document}